\documentclass[a4paper,10pt]{amsart}
\usepackage[margin=1in]{geometry}
\usepackage{enumerate, amsmath, amsfonts, amssymb, amsthm, mathtools, thmtools, wasysym, graphics, graphicx, xcolor, frcursive,xparse,comment,ytableau,stmaryrd,bbm,array,colortbl,tensor}
\usepackage[all]{xy}
\usepackage[boxed,norelsize]{algorithm2e}
\usepackage{hyperref}

\usepackage{url, hypcap}
\hypersetup{colorlinks=true, citecolor=darkblue, linkcolor=darkblue}

\usepackage{tikz}
\usetikzlibrary{calc,through,backgrounds,shapes,matrix}
\definecolor{darkblue}{rgb}{0.0,0,0.7} 
\newcommand{\darkblue}{\color{darkblue}} 
\definecolor{darkred}{rgb}{0.7,0,0} 
\definecolor{lightgrey}{rgb}{0.7,0.7,0.7} 

\definecolor{meet}{RGB}{255,205,111}
\definecolor{join}{RGB}{0,77,178}

\newtheorem{theorem}{Theorem}[section]
\newtheorem{proposition}[theorem]{Proposition}
\newtheorem{corollary}[theorem]{Corollary}
\newtheorem{lemma}[theorem]{Lemma}

\theoremstyle{definition}
\newtheorem{definition}[theorem]{Definition}
\newtheorem{example}[theorem]{Example}
\newtheorem{conjecture}[theorem]{Conjecture}

\newtheorem{remark}[theorem]{Remark}
\usepackage[procnames]{listings}
\usepackage[nameinlink]{cleveref}
\usepackage[all]{xy}
\usepackage[T1]{fontenc}
\crefformat{footnote}{#2\footnotemark[#1]#3}
\crefformat{conjecture}{Conjecture~#2#1#3}

\usepackage[cal=boondoxo]{mathalfa}
\usepackage[colorinlistoftodos]{todonotes}

\newcommand{\defn}[1]{\emph{\darkblue #1}}

\newcommand{\pop}{\mathsf{Pop}}

\newcommand{\spn}{\mathrm{span}}

\newcommand{\wo}{w_\circ}

\newcommand{\Des}{\mathrm{Des}}
\newcommand{\NC}{\mathrm{NC}}

\newcommand{\mov}{\mathrm{Mov}}
\newcommand{\im}{\mathrm{im}}
\newcommand{\aexc}{\mathrm{aexc}}
\newcommand{\Aexc}{\mathrm{Aexc}}
\newcommand{\cyc}{\mathrm{cyc}}
\newcommand{\unfold}{\mathrm{unfold}}
\newcommand{\fold}{\mathrm{fold}}

\newcommand{\U}{{W'}}
\newcommand{\uu}{{w'}}

\DeclarePairedDelimiter\abs{\lvert}{\rvert}

\usetikzlibrary{math}
\usepackage{xifthen}
\usepackage{xstring}
\SetKwInput{kwset}{set}
\usetikzlibrary{arrows,backgrounds,calc,trees}
\pgfdeclarelayer{background}
\pgfsetlayers{background,main}

\title{Coxeter Pop-Tsack Torsing}

\author[C.~Defant]{Colin Defant}
\address[C.~Defant]{Princeton University}
\email{cdefant@princeton.edu}

\author[N.~Williams]{Nathan Williams}
\address[N.~Williams]{University of Texas at Dallas}
\email{nathan.williams1@utdallas.edu}


\begin{document}

\begin{abstract}
Given a finite irreducible Coxeter group $W$ with a fixed Coxeter element $c$, we define the \emph{Coxeter pop-tsack torsing operator} $\mathsf{Pop}_T:W\to W$ by $\mathsf{Pop}_T(w)=w\cdot\pi_T(w)^{-1}$, where $\pi_T(w)$ is the join in the noncrossing partition lattice $\NC(w,c)$ of the set of reflections lying weakly below $w$ in the absolute order. This definition serves as a ``Bessis dual'' version of the first author's notion of a Coxeter pop-stack sorting operator, which, in turn, generalizes the pop-stack-sorting map on symmetric groups. We show that if $W$ is coincidental or of type~$D$, then the identity element of $W$ is the unique periodic point of $\mathsf{Pop}_T$ and the maximum size of a forward orbit of $\mathsf{Pop}_T$ is the Coxeter number $h$ of $W$. In each of these types, we obtain a natural lift from $W$ to the dual braid monoid of $W$. We also prove that $W$ is coincidental if and only if it has a unique forward orbit of size $h$. For arbitrary $W$, we show that the forward orbit of $c^{-1}$ under $\mathsf{Pop}_T$ has size $h$ and is isolated in the sense that none of the non-identity elements of the orbit have preimages lying outside of the orbit.  
\end{abstract}

\maketitle

\section{Introduction}

\subsection{Pop-stack sorting}
Suppose $X$ is a set and $f:X\to X$ is a function. Define the \defn{forward orbit} of an element $x\in X$ under the map $f$ to be the set \[O_f(x)=\{x,f(x),f^2(x),\ldots\},\] where $f^i$ denotes the $i^\text{th}$ iterate of $f$. One of the primary aims of \emph{dynamical combinatorics} is to understand the sizes of forward orbits of combinatorially-defined maps.

A rich source of combinatorial dynamical systems comes from the symmetric group $\mathfrak S_n$, which is the group of permutations of the set $[n]:=\{1,\ldots,n\}$. For example, the \defn{pop-stack sorting map} is the function $\pop:\mathfrak S_n\to \mathfrak S_n$ that acts by reversing all of the descending runs of a permutation while keeping different descending runs in the same order relative to each other. This map, which originally appeared in an article by Ungar concerning directions determined by points in the plane \cite{Ungar}, has gained interest in the past few years among enumerative combinatorialists \cite{AlbertVatter, Asinowski, Asinowski2, Elder, ClaessonPop, ClaessonPop2, Pudwell}. See \cite{defant2021stack} and the references therein for several other dynamical systems arising from sorting operators on permutations. 

It is not difficult to show that the identity permutation $e=12\cdots n$ is the unique fixed point of $\pop:\mathfrak S_n\to \mathfrak S_n$ and that every forward orbit of this map contains $e$. However, it is surprisingly difficult to determine the maximum size of a forward orbit; this was done by Ungar, who proved that $\max\limits_{w\in \mathfrak S_n}\abs{O_\pop(w)}=n$. This result was recently reproven by Albert and Vatter \cite{AlbertVatter}. 

\subsection{Coxeter pop-stack sorting}
Let $W$ be a finite Coxeter group. Let $S$ be a set of simple reflections of $W$, and let $\wo$ be the longest element of $W$.  Let $\leq_S$ denote the left weak order on $W$. For $J\subseteq S$, let $\wo(J)$ be the longest element of the parabolic subgroup of $W$ generated by $J$. Define the \defn{longest-element projection} from $W$ to the set of parabolic long elements by
\begin{align}
\pi_S:W &\to \big\{\wo(J) : J \subseteq S \big\} \nonumber\\
\pi_S(w)&=\bigvee_{s \leq_S w}^{\mathrm{Weak}} s.
\end{align}
That is, $\pi_S(w)$ is the join in left weak order of the set of simple reflections lying weakly below $w$. Alternatively, $\pi_S(w)=\wo(\Des_R(w))$ is the longest element of the parabolic subgroup generated by the right descent set $\Des_R(w)$ of $w$.

A natural generalization of the pop-stack sorting map to Coxeter groups was introduced in~\cite{defant2021stack} as 
\begin{align}
\pop_S&:W\to W \nonumber\\
\pop_S(w) &= w \cdot \pi_S(w)^{-1}.
\label{eq:pops}
\end{align} Indeed, as discussed in \cite{defant2021stack}, this operator agrees with the pop-stack sorting map when $W$ is the symmetric group $\mathfrak S_n$. Observe that $\pop_S$ fixes the identity element $e\in W$. On the other hand, if $w\in W$ is not the identity element, then $\pop_S(w)$ is strictly less than $w$ in the right weak order.  When $W$ is replaced by its braid group, the elements of the forward orbit of $w$ under $\pop_S$ are the prefixes of the Brieskorn normal form of $w$~\cite{brieskorn1971fundamentalgruppe,brieskorn1972artin}.

One of the main results from \cite{defant2021stack} is the following generalization of Ungar's theorem to arbitrary finite irreducible Coxeter groups. 

\begin{theorem}[{\cite{defant2021stack}}]\label{ColinThm}
Let $W$ be a finite irreducible Coxeter group with a set $S$ of simple reflections and with Coxeter number $h$. Then $\max\limits_{w\in W}\left|O_{\pop_S}(w)\right|=h$.
\end{theorem}

\Cref{fig:b3} illustrates the action of $\pop_S$ on the hyperoctahedral group $B_3$. As predicted by \Cref{ColinThm}, the maximum size of a forward orbit in this diagram is $6$, which is the Coxeter number of $B_3$.

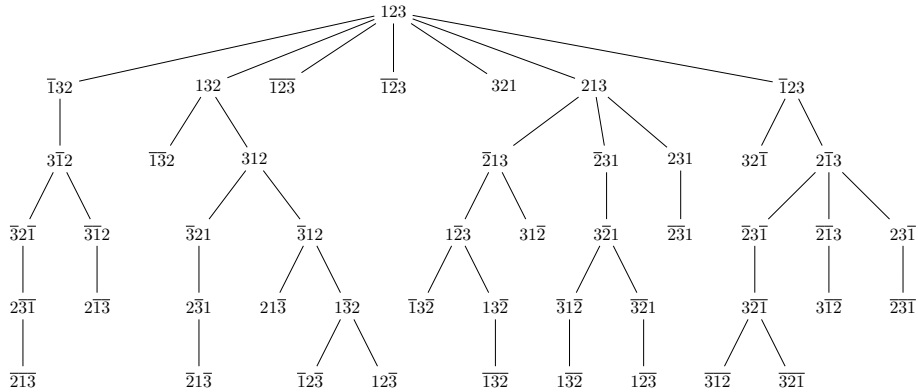
\begin{figure}[htbp]
\begin{center}
\scalebox{.65}{
\begin{tikzpicture}[scale=1.5]
\node (n1n2n3) at (5.0, -1) {$\overline{1}\overline{2}\overline{3}$};
\node (n1n23) at (6.5, -1) {$\overline{1}\overline{2}3$};
\node (n2n1n3) at (1.5, -5) {$\overline{2}\overline{1}\overline{3}$};
\node (2n3n1) at (1.5, -4) {$2\overline{3}\overline{1}$};
\node (n32n1) at (1.5, -3) {$\overline{3}2\overline{1}$};
\node (2n1n3) at (2.5, -4) {$2\overline{1}\overline{3}$};
\node (n3n12) at (2.5, -3) {$\overline{3}\overline{1}2$};
\node (3n12) at (2.0, -2) {$3\overline{1}2$};
\node (n132) at (2.0, -1) {$\overline{1}32$};
\node (321) at (8.0, -1) {$321$};
\node (n1n32) at (3.375, -2) {$\overline{1}\overline{3}2$};
\node (n21n3) at (3.875, -5) {$\overline{2}1\overline{3}$};
\node (2n31) at (3.875, -4) {$2\overline{3}1$};
\node (n321) at (3.875, -3) {$\overline{3}21$};
\node (21n3) at (4.875, -4) {$21\overline{3}$};
\node (n12n3) at (5.375, -5) {$\overline{1}2\overline{3}$};
\node (12n3) at (6.375, -5) {$12\overline{3}$};
\node (1n32) at (5.875, -4) {$1\overline{3}2$};
\node (n312) at (5.375, -3) {$\overline{3}12$};
\node (312) at (4.625, -2) {$312$};
\node (132) at (4.0, -1) {$132$};
\node (n13n2) at (6.875, -4) {$\overline{1}3\overline{2}$};
\node (n1n3n2) at (7.875, -5) {$\overline{1}\overline{3}\overline{2}$};
\node (13n2) at (7.875, -4) {$13\overline{2}$};
\node (1n23) at (7.375, -3) {$1\overline{2}3$};
\node (31n2) at (8.375, -3) {$31\overline{2}$};
\node (n213) at (7.875, -2) {$\overline{2}13$};
\node (1n3n2) at (8.875, -5) {$1\overline{3}\overline{2}$};
\node (n31n2) at (8.875, -4) {$\overline{3}1\overline{2}$};
\node (1n2n3) at (9.875, -5) {$1\overline{2}\overline{3}$};
\node (n3n21) at (9.875, -4) {$\overline{3}\overline{2}1$};
\node (3n21) at (9.375, -3) {$3\overline{2}1$};
\node (n231) at (9.375, -2) {$\overline{2}31$};
\node (n2n31) at (10.375, -3) {$\overline{2}\overline{3}1$};
\node (231) at (10.375, -2) {$231$};
\node (213) at (9.208333333333334, -1) {$213$};
\node (32n1) at (11.375, -2) {$32\overline{1}$};
\node (n3n1n2) at (10.875, -5) {$\overline{3}\overline{1}\overline{2}$};
\node (n3n2n1) at (11.875, -5) {$\overline{3}\overline{2}\overline{1}$};
\node (3n2n1) at (11.375, -4) {$3\overline{2}\overline{1}$};
\node (n23n1) at (11.375, -3) {$\overline{2}3\overline{1}$};
\node (3n1n2) at (12.375, -4) {$3\overline{1}\overline{2}$};
\node (n2n13) at (12.375, -3) {$\overline{2}\overline{1}3$};
\node (n2n3n1) at (13.375, -4) {$\overline{2}\overline{3}\overline{1}$};
\node (23n1) at (13.375, -3) {$23\overline{1}$};
\node (2n13) at (12.375, -2) {$2\overline{1}3$};
\node (n123) at (11.875, -1) {$\overline{1}23$};
\node (123) at (6.5, 0) {$123$};

\draw (n3n2n1) -- (3n2n1);
\draw (n3n21) -- (1n2n3);
\draw (n3n21) -- (3n21);
\draw (n3n1n2) -- (3n2n1);
\draw (n3n12) -- (2n1n3);
\draw (n3n12) -- (3n12);
\draw (n31n2) -- (1n3n2);
\draw (n31n2) -- (3n21);
\draw (n312) -- (1n32);
\draw (n312) -- (21n3);
\draw (n312) -- (312);
\draw (n32n1) -- (2n3n1);
\draw (n32n1) -- (3n12);
\draw (n321) -- (2n31);
\draw (n321) -- (312);
\draw (n2n3n1) -- (23n1);
\draw (n2n31) -- (231);
\draw (n2n1n3) -- (2n3n1);
\draw (n2n13) -- (2n13);
\draw (n2n13) -- (3n1n2);
\draw (n21n3) -- (2n31);
\draw (n213) -- (1n23);
\draw (n213) -- (213);
\draw (n213) -- (31n2);
\draw (n23n1) -- (2n13);
\draw (n23n1) -- (3n2n1);
\draw (n231) -- (213);
\draw (n231) -- (3n21);
\draw (n1n3n2) -- (13n2);
\draw (n1n32) -- (132);
\draw (n1n2n3) -- (123);
\draw (n1n23) -- (123);
\draw (n12n3) -- (1n32);
\draw (n123) -- (123);
\draw (n123) -- (2n13);
\draw (n123) -- (32n1);
\draw (n13n2) -- (1n23);
\draw (n132) -- (123);
\draw (n132) -- (3n12);
\draw (1n32) -- (12n3);
\draw (1n23) -- (13n2);
\draw (123) -- (132);
\draw (123) -- (213);
\draw (123) -- (321);
\draw (132) -- (312);
\draw (2n13) -- (23n1);
\draw (213) -- (231);
\end{tikzpicture}}
\end{center}
\caption{A diagram of the Coxeter pop-stack sorting operator $\pop_S:B_3\to B_3$. Each vertex in the tree is an element of $B_3$, represented as a signed permutation. The root vertex is the identity element $e$; the parent of each non-root vertex $w$ is $\pop_S(w)$. }
\label{fig:b3}
\end{figure}


Let us also remark that, by construction, the elements of $W$ that get sorted into the identity $e$ in at most one iteration of $\pop_S$ are exactly the longest elements of standard parabolic subgroups---hence, there are $2^{|S|}$ such elements.

\subsection{Coxeter pop-tsack torsing}
Let $T$ be the set of all reflections of $W$, and fix a Coxeter element $c$ of $W$ (a regular element of order $h$).  Let $\leq_T$ denote the absolute order on $W$. There is a general ``dual'' philosophy for finite Coxeter groups~\cite{bessis2003dual} in which the set $S$ of simple reflections is replaced by the set $T$ of all reflections, the longest element $\wo$ is replaced by the Coxeter element $c$, and the (left) weak order is replaced by the noncrossing partition lattice $\NC(W,c)$ (the interval $[e,c]$ in absolute order).

Define the \defn{noncrossing projection}
\begin{align}\label{EqNCProj}
\pi_T(\cdot,c)&:W \to \NC(W,c) \nonumber\\
\pi_T(w,c) &= \bigvee_{t \leq_T w}^{\NC(W,c)} t.
\end{align}
That is, $\pi_T(w,c)$ is the join in the noncrossing partition lattice $\NC(W,c)$ of the set of reflections lying weakly below $w$ in absolute order.  It is natural to make the following definition as the ``dual'' of $\pop_S$ from~\eqref{eq:pops}:

\begin{definition}
\label{def:dual_pop_stack}
The \defn{Coxeter pop-tsack torsing operator}\footnote{We have interchanged the roles of $S$ and $T$.} is the map
\begin{align*}
 \pop_T(\cdot,c)&:W\to W \nonumber \\
\pop_T(w,c) &= w \cdot \pi_T(w,c)^{-1}.
\end{align*}
\end{definition}

When $c$ is understood, we may abbreviate $\pi_T(w,c)$ and $\pop_T(w,c)$ as $\pi_T(w)$ and $\pop_T(w)$, respectively; we establish in~\Cref{sec:general_properties} many structural properties of $\pop_T$, including that its orbit structure does not depend on $c$.  
By construction, the elements $w\in W$ such that $\pop_T(w)=e$ are exactly the noncrossing partitions---hence, the number of such elements is the \defn{$W$-Catalan number} \[\left|\NC(W,c)\right|=\prod_{i=1}^n \frac{h+d_i}{d_i},\] where $d_1,\ldots,d_n$ are the degrees of $W$.

\Cref{fig:b3dual} illustrates the action of $\pop_T$ on $B_3$; additional data for other types is given in~\Cref{fig:data}.  Notice that every forward orbit of $\pop_T:B_3\to B_3$ contains the identity element $e$, which is a fixed point.  
Upon inspection of \Cref{fig:b3dual}, we see that there is a unique forward orbit of size $6$ and that all other forward orbits have size at most $5$. This special forward orbit of size $6$ is $O_{\pop_T}(c^{-1})=\{c^{-1},c^{-2},\ldots,c^{-5},e\}$. Furthermore, this forward orbit is ``isolated'' in the sense that for each $2\leq i\leq 5$, the only preimage of $c^{-i}$ under $\pop_T$ is $c^{-(i-1)}$. In~\Cref{sec:inverse_c_orbit}, we use uniform methods to show that this phenomenon holds for all finite irreducible Coxeter groups.

\begin{figure}[htbp]
\begin{center}
\scalebox{.55}{
\begin{tikzpicture}[scale=1.5]
\node (n3n1n2) at (4.75, -2) {$\overline{3}\overline{1}\overline{2}$};
\node (n31n2) at (5.75, -2) {$\overline{3}1\overline{2}$};
\node (2n3n1) at (5.25, -1) {$2\overline{3}\overline{1}$};
\node (n32n1) at (8.5, -1) {$\overline{3}2\overline{1}$};
\node (21n3) at (9.5, -1) {$21\overline{3}$};
\node (n3n21) at (2.75, -4) {$\overline{3}\overline{2}1$};
\node (n1n3n2) at (2.75, -3) {$\overline{1}\overline{3}\overline{2}$};
\node (2n1n3) at (2.75, -2) {$2\overline{1}\overline{3}$};
\node (2n31) at (3.75, -3) {$2\overline{3}1$};
\node (n321) at (4.75, -3) {$\overline{3}21$};
\node (n12n3) at (4.25, -2) {$\overline{1}2\overline{3}$};
\node (32n1) at (3.5, -1) {$32\overline{1}$};
\node (n2n3n1) at (5.75, -3) {$\overline{2}\overline{3}\overline{1}$};
\node (1n32) at (6.75, -3) {$1\overline{3}2$};
\node (1n2n3) at (6.25, -2) {$1\overline{2}\overline{3}$};
\node (n1n32) at (7.75, -4) {$\overline{1}\overline{3}2$};
\node (n2n1n3) at (7.75, -3) {$\overline{2}\overline{1}\overline{3}$};
\node (3n2n1) at (7.75, -2) {$3\overline{2}\overline{1}$};
\node (13n2) at (7.0, -1) {$13\overline{2}$};
\node (1n23) at (9.0, -1) {$1\overline{2}3$};
\node (n132) at (10, -1) {$\overline{1}32$};
\node (321) at (10.5, -1) {$321$};
\node (132) at (11.0, -1) {$132$};
\node (213) at (12.0, -1) {$213$};
\node (n3n2n1) at (12.5, -1) {$\overline{3}\overline{2}\overline{1}$};
\node (1n3n2) at (13.5, -1) {$1\overline{3}\overline{2}$};
\node (n3n12) at (14.5, -2) {$\overline{3}\overline{1}2$};
\node (3n12) at (15.5, -2) {$3\overline{1}2$};
\node (n23n1) at (15.0, -1) {$\overline{2}3\overline{1}$};
\node (n2n13) at (13, -1) {$\overline{2}\overline{1}3$};
\node (12n3) at (14, -1) {$12\overline{3}$};
\node (n312) at (16.0, -5) {$\overline{3}12$};
\node (n2n31) at (16.0, -4) {$\overline{2}\overline{3}1$};
\node (n1n2n3) at (16.0, -3) {$\overline{1}\overline{2}\overline{3}$};
\node (3n1n2) at (16.0, -2) {$3\overline{1}\overline{2}$};
\node (23n1) at (16.0, -1) {$23\overline{1}$};
\node (n21n3) at (17.0, -4) {$\overline{2}1\overline{3}$};
\node (3n21) at (17.0, -3) {$3\overline{2}1$};
\node (n13n2) at (17.0, -2) {$\overline{1}3\overline{2}$};
\node (n213) at (18.0, -3) {$\overline{2}13$};
\node (n231) at (19.0, -3) {$\overline{2}31$};
\node (n1n23) at (18.5, -2) {$\overline{1}\overline{2}3$};
\node (2n13) at (17.75, -1) {$2\overline{1}3$};
\node (n123) at (11.5, -1) {$\overline{1}23$};
\node (31n2) at (19.5, -2) {$31\overline{2}$};
\node (312) at (20.5, -2) {$312$};
\node (231) at (20.0, -1) {$231$};
\node (123) at (11.5, 0) {$123$};

\draw (n3n2n1) -- (123);
\draw (n3n21) -- (n1n3n2);
\draw (n3n1n2) -- (2n3n1);
\draw (n3n12) -- (n23n1);
\draw (n31n2) -- (2n3n1);
\draw (n312) -- (n2n31);
\draw (n32n1) -- (123);
\draw (n321) -- (n12n3);
\draw (n2n3n1) -- (1n2n3);
\draw (n2n31) -- (n1n2n3);
\draw (n2n1n3) -- (n1n32);
\draw (n2n1n3) -- (3n2n1);
\draw (n2n13) -- (123);
\draw (n21n3) -- (3n21);
\draw (n213) -- (n1n23);
\draw (n23n1) -- (123);
\draw (n23n1) -- (3n12);
\draw (n231) -- (n1n23);
\draw (n1n3n2) -- (2n1n3);
\draw (n1n2n3) -- (3n1n2);
\draw (n1n23) -- (2n13);
\draw (n12n3) -- (2n31);
\draw (n12n3) -- (32n1);
\draw (n123) -- (123);
\draw (n13n2) -- (2n13);
\draw (n13n2) -- (3n21);
\draw (n132) -- (123);
\draw (1n3n2) -- (123);
\draw (1n32) -- (1n2n3);
\draw (1n2n3) -- (13n2);
\draw (1n23) -- (123);
\draw (12n3) -- (123);
\draw (123) -- (13n2);
\draw (123) -- (132);
\draw (123) -- (2n3n1);
\draw (123) -- (2n13);
\draw (123) -- (21n3);
\draw (123) -- (213);
\draw (123) -- (23n1);
\draw (123) -- (231);
\draw (123) -- (32n1);
\draw (123) -- (321);
\draw (13n2) -- (3n2n1);
\draw (2n1n3) -- (32n1);
\draw (23n1) -- (3n1n2);
\draw (231) -- (31n2);
\draw (231) -- (312);
\end{tikzpicture}}
\end{center}
\caption{A diagram of the Coxeter pop-tsack torsing operator $\pop_T:B_3\to B_3$ using the Coxeter element $c=23\overline{1}$, which is $(\overline 1\,\,\overline 2\,\,\overline 3\,\,1\,\,2\,\,3)$ in cycle notation (bars denote negative signs). The root vertex of the tree is the identity element $e$; the parent of each non-root vertex $w$ is $\pop_T(w,c)$.}
\label{fig:b3dual}
\end{figure}
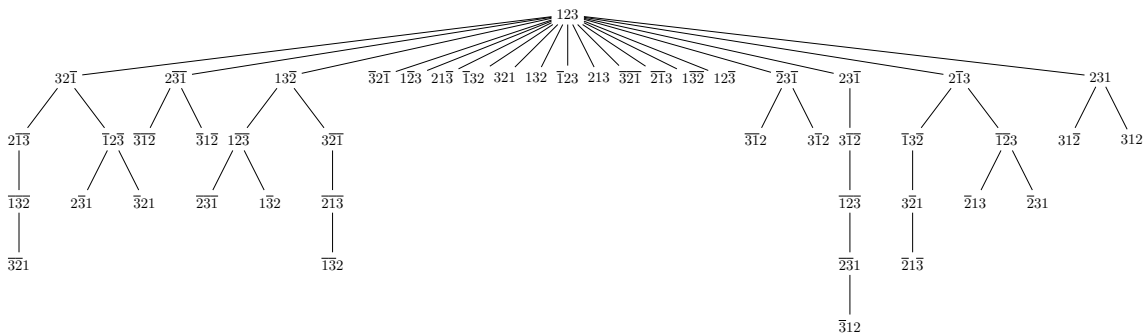

{
\renewcommand{\thetheorem}{\ref{thm:cinverseorbit}}
\begin{theorem}
The forward orbit of $c^{-1}$ under $\pop_T(\cdot,c)$ is \[O_{\pop_T}(c^{-1})=\{c^{-1}, c^{-2},\ldots, c^{-(h-1)}, e\}.\] The element $c^{-1}$ has no preimages under $\pop_T$. Moreover, if $2\leq i\leq h-1$, then the only preimage of $c^{-i}$ under $\pop_T$ is $c^{-(i-1)}$. 
\end{theorem}
\addtocounter{theorem}{-1}
}

Recall that a Coxeter group is called \defn{coincidental} if it is of type~$A$, type~$B$, type~$H_3$, or dihedral type.  We use case-by-case methods to understand the maximum possible size of a forward orbit of $\pop_T$ in the coincidental types (\Cref{sec:dihedral} for the dihedral types, \Cref{sec:typea} for type $A$, and \Cref{sec:typeb} for type $B$; the data for $H_3$ is in~\Cref{fig:data}).

{
\renewcommand{\thetheorem}{\ref{thm:main1}}
\begin{theorem}
Let $W$ be a Coxeter group of coincidental type with a fixed Coxeter element $c$ and Coxeter number $h$. Then 
\begin{itemize}
    \item $\pop_T^{h-1}(w)=e$ for all $w\in W$,
    \item  $\max\limits _{w\in W}\left|O_{\pop_T}(w)\right|=h$, and
    \item the only forward orbit of size $h$ is $O_{\pop_T}(c^{-1})$. 
\end{itemize}
\end{theorem}
\addtocounter{theorem}{-1}
}

Beyond the coincidental types, a few things can go awry.  The first unexpected phenomenon is that there can be more than one forward orbit of size $h$.  In type $D_n$, for example, the maximum size of a forward orbit is $h=2n-2$---but there may be several such orbits ($D_4$ has seven forward orbits of size $6$).  In~\Cref{sec:typed}, we prove the following weaker version of \Cref{thm:main1}.

{
\renewcommand{\thetheorem}{\ref{thm:main2}}
\begin{theorem}
Let $W$ be a Coxeter group of type $D_n$. Then 
\begin{itemize}
    \item $\pop_T^{2n-3}(w)=e$ for all $w\in W$, and
    \item  $\max\limits _{w\in W}\left|O_{\pop_T}(w)\right|=2n-2$.
\end{itemize}
\end{theorem}
\addtocounter{theorem}{-1}
}

\Cref{conj:A} and \Cref{conj:B} give conjectural formulas for the number of elements with forward orbits of size $h$ or $h-1$ when $W$ is of type $A$ or $B$. \Cref{conj:D} gives a conjectural formula for the number of elements whose forward orbits have size $h$ when $W$ is of type $D$. 

Things go even more haywire beyond types $A,B,D,H_3,$ and $I_2(m)$: not only do the remaining types have forward orbits of size greater than $h$, but they have periodic orbits that do not contain the identity.  This contrasts with $\pop_S$, which always moves each non-identity element strictly downward in the right weak order. 



Call an element $x\in W$ \defn{periodic} if $\pop_T^k(x)=x$ for some $k\geq 1$; in this case, we also say the forward orbit $O_{\pop_T}(x)$ is \defn{periodic}. In general, a finite irreducible Coxeter group can contain periodic orbits other than that of the identity element.

{
\renewcommand{\thetheorem}{\ref{thm:periodic}}
\begin{theorem}
The map $\pop_T$ has periodic orbits of size $h$ for $W$ of types $E_6$, $E_7$, $E_8$, $F_4$, or $H_4$. 
\end{theorem}
\addtocounter{theorem}{-1}
}

\begin{figure}[htbp]
\begin{center}
\scalebox{0.8}{
\begin{tabular}{c|cccccccccccccccc|c}
Type & 0 & 1 & 2 & 3 & 4 & 5 & 6 & 7 & 8 & 9 & 10 & 11 & 12 & 13 & 14 & 15 & $\infty$\\ \hline
$A_2$ & 1 & 4 & 1  & & & & & & & & & & & & & &\\
$A_3$ & 1 & 13 & 9 & 1  & & & & & & & & & & & & &\\
$A_4$ & 1 & 41 & 56 & 21 & 1 & & & & & & & & & & & &\\
$A_5$ & 1 & 131 & 305 & 234 & 48 & 1  & & & & & & & & & & &\\ \hline
$B_2$ & 1 & 5 & 1 & 1  & & & & & & & & & & & & &\\
$B_3$ & 1 & 19 & 13 & 10 & 4 & 1  & & & & & & & & & & &\\
$B_4$ & 1 & 69 & 101 & 91 & 61 & 49 & 11 & 1  & & & & & & & & &\\
$B_5$ & 1 & 251 & 646 & 816 & 686 & 761 & 466 & 186 & 26 & 1 & & & & & & & \\ \hline
$D_4$ & 1 & 49 & 85 & 34 & 15 & 7 & & & & & & & & & & &\\
$D_5$ & 1 & 181 & 565 & 523 & 301 & 217 & 107 & 25 & & & & & & & & &\\
$D_6$ & 1 & 671 & 3336 & 5396 & 4416 & 3641 & 2946 & 2026 & 536 & 71 & & & & & & &\\
$D_7$ & 1 & 2507 & 18872 & 45274 & 55701 & 50960 & 50835 & 50643 & 32080 & 13193 & 2313 & 181 & & & & &\\ \hline
$F_4$ & 1 & 104 & 171 & 194 & 191 & 119 & 71 & 71 & 71 & 59 & 59 & 11 & 3 & 3 & & & 24 \\
$E_6$ & 1 & 832 & 4619 & 8214 & 7843 & 8039 & 7307 & 4835 & 3407 & 2687 & 2423 & 1055 & 371 & 107 & 68 & 8 & 24 \\ \hline
$H_3$ & 1 & 31 & 21 & 16 & 21 & 11 & 6 & 6 & 6 & 1 & & & & & & &\\
\end{tabular}}
\end{center}
\caption{Some data for the map $\pop_T$.  Columns labeled by $i$ record the number of elements requiring $i$ iterations of $\pop_T$ to reach the identity.  The column labeled by $\infty$ represents elements that lie in periodic orbits other than the orbit $O_{\pop_T}(e)=\{e\}$.}
\label{fig:data}
\end{figure}


From~\Cref{thm:main1,thm:main2,thm:periodic}, we conclude the following characterization of coincidental Coxeter groups.
\begin{corollary}
\label{cor:main4}
A finite irreducible Coxeter group $W$ with Coxeter number $h$ is coincidental if and only if $\pop_T:W\to W$ has a unique forward orbit of size $h$.
\end{corollary}

\subsection{Normal forms}
Because a presentation of the braid group of a finite Coxeter group is given by simply forgetting that simple reflections square to the identity, the Coxeter group embeds canonically into the positive braid monoid by simply lifting reduced expressions (in simple reflections).  It has been an open problem to find a similar ``canonical'' embedding in the dual braid monoid.\footnote{We thank T.~Gobet for bringing this problem to our attention.}  The pop-tsack torsing operator gives a solution to this problem for those finite Coxeter groups in which the only periodic orbit under $\pop_T$ is that of the identity---that is, for coincidental Coxeter groups and for type~$D$.

 {
\renewcommand{\thetheorem}{\ref{thm:normal_form}}
\begin{theorem}
Let $W$ be a finite irreducible Coxeter group whose only periodic orbit under $\pop_T$ is that of the identity.  For $w\in W$, write $w_1=w$ and $w_i=\pop_T(w_{i-1},c)$ for $i\geq 2$. Let $k=\lvert O_{\pop_T}(w)\rvert-1$ so that (assuming $w\neq e$) $w_k\neq w_{k+1}=e$.  Then the factorization \[w=\pi_T(w_k,c)\pi_T(w_{k-1},c)\cdots\pi_T(w_1,c)\] gives a lift from $W$ to the dual braid monoid of $W$.
\end{theorem}
\addtocounter{theorem}{-1}
}

It would be interesting to study the properties of this lift or to extend it to complex reflection groups (one potential issue is that for such groups, some reflections may not be noncrossing partitions).

We conclude with a generalization of stabilized-interval free permutations (as in \cite{callan2004counting,blitvic2014stabilized}) to finite Coxeter groups using our noncrossing projection map.

\section{Preliminaries}
Throughout this article, we assume that $W$ is a finite irreducible Coxeter group with a set $S$ of simple reflections. The \defn{rank} of $W$ is $|S|$. Given $g,w\in W$, we write $w^g$ for the conjugate $gwg^{-1}$ of $w$ by $g$. Let $T=\{s^w:s\in S, w\in W\}$ denote the set of \defn{reflections} of $W$. A \defn{$T$-word} for an element $w\in W$ is a word $t_1\cdots t_r$ over the alphabet $T$ that is equal to $w$ when considered as an element of $W$. The \defn{reflection length} of $w$, denoted $\ell_T(w)$, is the minimum length of a $T$-word for $w$. The \defn{absolute order} on $W$, denoted $\leq_T$, is defined by saying $v\leq_T w$ if $\ell_T(wv^{-1})=\ell_T(w)-\ell_T(v)$. Since $T$ is closed under conjugation by elements of $W$, we have $v\leq_T w$ if and only if $\ell_T(v^{-1}w)=\ell_T(w)-\ell_T(v)$. 

It is often useful to consider the \defn{geometric representation} of $W$ on a vector space $V\cong\mathbb R^{|S|}$. We refer the reader to \cite{armstrong2009generalized, BjornerBrenti} for more thorough discussions of this representation; here, we simply mention the basic results that we will need. We denote the action of an element $w\in W$ on a vector $\alpha\in V$ by $w\cdot\alpha$. For every $t\in T$, there is a hyperplane $H_t\subseteq V$ such that $t$ acts via the linear transformation that reflects through $H_t$. Furthermore, $w\cdot H_t=H_{t^w}$ for every $w\in W$ and $t\in T$. Define the \defn{moved space} of $w$ to be $\mov(w)=\im(w-1)$, the image of the linear transformation $w-1:V\to V$ ($1$ is the identity transformation). Brady and Watt \cite{BradyWatt} showed that if $t\in T$ and $w\in W$, then \begin{equation}\label{EqBradyWatt}
\mov(t)\subseteq\mov(w)\text{ if and only if }t\leq _T w.
\end{equation}

The ring $\mathrm{Sym}(V_\mathbb{C}^*)^W$ of $W$-invariant polynomials is a polynomial ring $\mathbb{C}[f_1,\ldots,f_n]$ over a set of invariant polynomials $f_i$ that may be chosen homogeneous and whose degrees $d_1 \leq \cdots \leq d_n$ are uniquely determined.  We write $h=d_n$ for the \defn{Coxeter number} of $W$.

A \defn{standard Coxeter element} of $W$ is an element obtained by multiplying the simple reflections together in some order.  A \defn{regular element} of $W$ is an element with an eigenvector that lies in the complement of the reflecting hyperplanes.  Following~\cite{reiner2017non}, we define a \defn{Coxeter element} of $W$ to be a regular element of order $h$.  For Weyl groups, Coxeter elements coincide with conjugates of standard Coxeter elements.  
More generally, a \defn{reflection automorphism} of $W$ is a group automorphism that preserves the set of reflections of $W$.  By~\cite[Proposition 1.4]{reiner2017non}, the reflection automorphisms of $W$ act transitively on its Coxeter elements---thus, all Coxeter elements have the same order $h$.

We assume throughout this article that we have fixed a Coxeter element $c$ of $W$. A \defn{noncrossing partition} is an element $v\in W$ such that $v\leq_T c$. The interval $[e,c]$ in the absolute order is called the \defn{noncrossing partition lattice} and denoted $\NC(W,c)$; it is indeed a lattice. This is the only type
of lattice we will consider in the rest of the article, so all meets (denoted by $\bigwedge$) and joins (denoted by $\bigvee$) will be with respect to this lattice. Furthermore, the definitions of the noncrossing projection $\pi_T$ and the Coxeter pop-tsack torsing operator $\pop_T$ from \eqref{EqNCProj} and \Cref{def:dual_pop_stack} are made with respect to this 
fixed Coxeter element. Since the reflection automorphisms of $W$ act transitively on the Coxeter elements, different Coxeter elements give rise to isomorphic noncrossing partition lattices.

\section{Dihedral Groups}
\label{sec:dihedral}

As the dihedral case will be useful when analyzing the general case, we immediately dispense with it in this short section.  The dihedral group $I_2(h)$ has $2h$ elements: the identity $e$, $h$ reflections, a Coxeter element $c$, and then the elements $c^{-1},c^{-2},\ldots,c^{-(h-1)}$.  The identity, reflections, and Coxeter element are all noncrossing partitions, and they all map to the identity under a single application of $\pop_T$.  For each $1\leq i\leq h-1$, the element $c^{-i}$ requires exactly $h-i$ iterations of $\pop_T$ to reach the identity.  The right-hand side of~\Cref{fig:arr_coxplane} illustrates the hyperplane arrangement for the dihedral group $I_2(4)$.

\begin{figure}[htbp]
  \begin{center}
      \includegraphics[width=.3\linewidth]{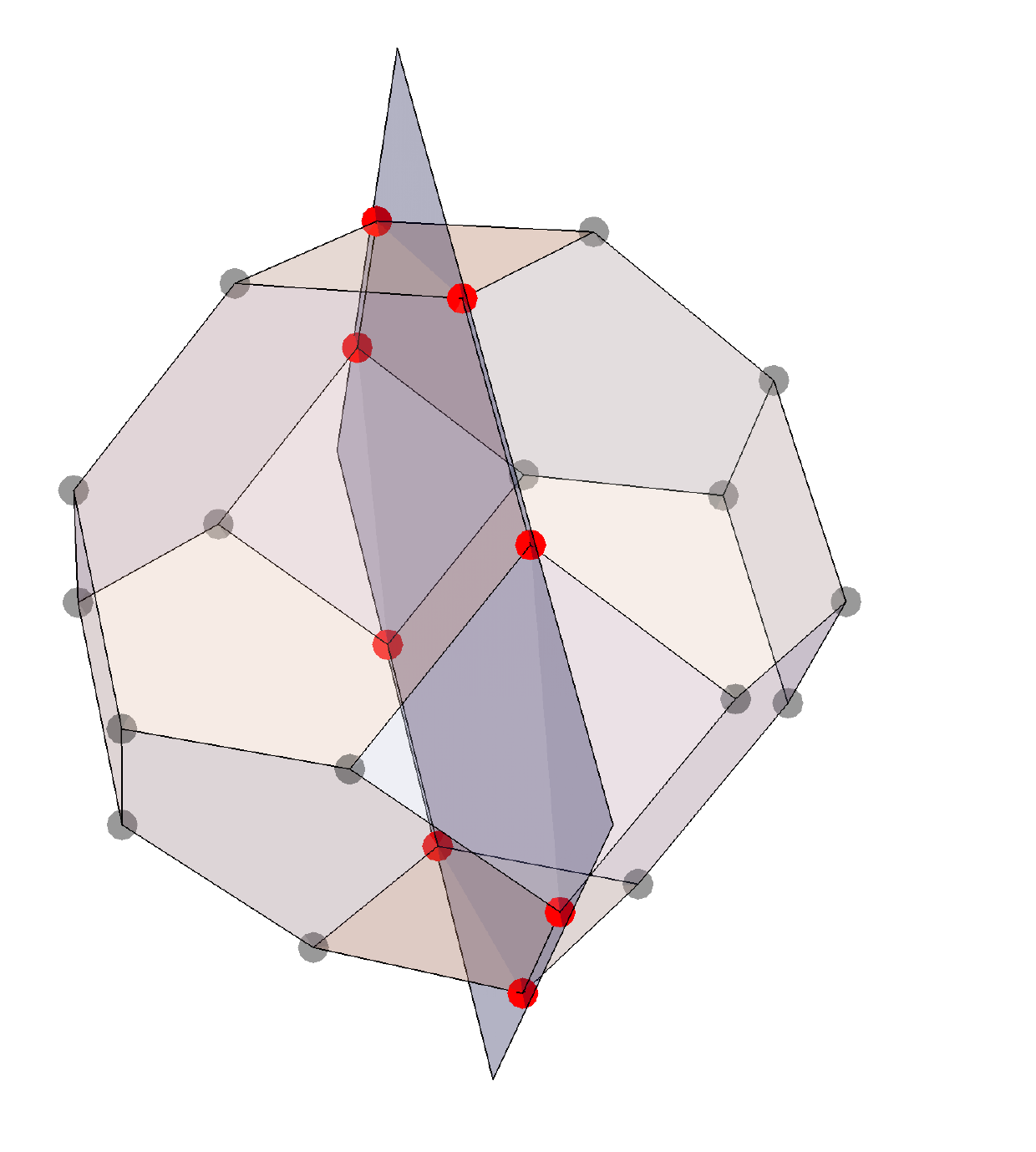} \hspace{2em}
      \begin{tikzpicture}[scale=1]
       \draw (-2,0) -- (2,0);
       \draw (0,-2) -- (0,2);
       \draw (-2,-2) -- (2,2);
       \draw (-2,2) -- (2,-2);
       \node at (1.5,.5) {$\color{red}{e}$};
       \node at (2.7,0) {$(12)(34)$};
       \node at (1.5,-.5) {$\color{red}{s}$};
       \node at (.5,1.5) {$\color{red}{t}$};
       \node at (2.3,2.3) {$(23)$};
       \node at (-.5,1.5) {$\color{red}{c}$};
       \node at (0,2.3) {$(13)(24)$};
       \node at (.5,-1.5) {$\color{red}{c^{3}}$};
       \node at (2.3,-2.3) {$(14)$};
       \node at (-.5,-1.5) {$\color{red}{sts}$};
       \node at (-1.5,.5) {$\color{red}{tst}$};
       \node at (-1.5,-.5) {$\color{red}{c^2}$};
      \end{tikzpicture}
    \end{center}
\caption{{\it Left:} the permutahedron for $\mathfrak{S}_4$ with the eight vertices on the Coxeter plane for bipartite $c$ labeled in red.  {\it Right:} the hyperplane arrangement for the corresponding dihedral group $I_2(4)$.}
\label{fig:arr_coxplane}
\end{figure}
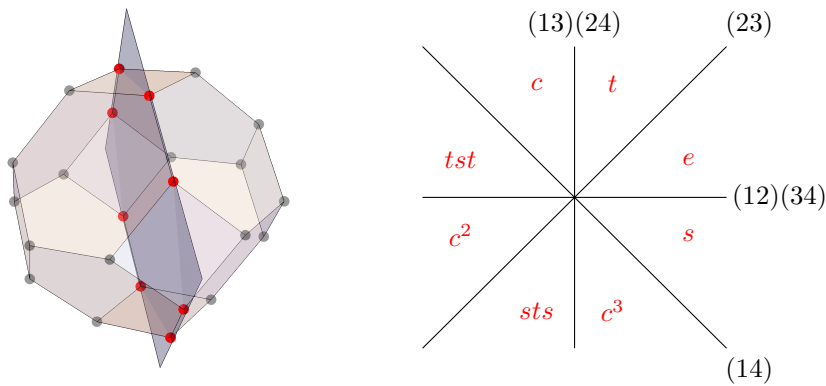

\section{General Properties of $\pop_T$}
\label{sec:general_properties}

\subsection{Absolute monotonicity}

Given $x\in W$, let $W_x$ be the subgroup of $W$ generated by the reflections lying weakly below $x$ in the absolute order. The subgroups $W_x$ for $x\in \NC(W,c)$ are called the \defn{noncrossing parabolic subgroups} of $W$. As discussed in \cite{ReadingNoncrossing}, the noncrossing parabolic subgroups of $W$ form a lattice under the containment order, and the map $x\mapsto W_x$ gives an isomorphism from $\NC(W,c)$ to this lattice. 

\begin{lemma}
\label{lem:monotone}
For every $x \in W$, we have $\pi_T(\pop_T(x))\leq_T \pi_T(x)$.
\end{lemma}
\begin{proof}
The parabolic subgroup generated by the reflections lying weakly below $\pi_T(x)$ in absolute order is $W_{\pi_T(x)}$. All of the reflections lying weakly below $x$ also lie weakly below $\pi_T(x)$, so $x$ is an element of $W_{\pi_T(x)}$. Of course, $\pi_T(x)$ also belongs to $W_{\pi_T(x)}$. This means that $\pop_T(x)=x\pi_T(x)^{-1}$ is in $W_{\pi_T(x)}$, so all of the reflections weakly below $\pop_T(x)$ are in $W_{\pi_T(x)}$. It follows that 
$W_{\pi_T(\pop_T(x))}$ is contained in $W_{\pi_T(x)}$, and this proves the result. 
\end{proof}

The following proposition gives a condition for elements to have an ``isolated'' forward orbit under $\pop_T$. It states that the only possible preimage under $\pop_T$ of an element $w \in W$ with $\pi_T(w)=c$ is $wc$.

\begin{proposition}
\label{prop:preimages}
If $x,w\in W$ are such that $\pop_T(x)=w$ and $\pi_T(w)=c$, then $x=wc$.
\end{proposition}

\begin{proof}
Suppose $x$ is such that $\pop_T(x)=w$.  By \Cref{lem:monotone}, $c\leq_T\pi_T(x)$, so we must have $\pi_T(x)=c$. However, this implies that $x=\pop_T(x)\pi_T(x)=wc$.  
\end{proof}

\subsection{Equivariance under reflection automorphisms}

In this section, we make a point of decorating $\pi_T$ and $\pop_T$ by the Coxeter element $c$ in order to differentiate different choices of Coxeter elements.  We show a sort of equivariance of $\pop_T$ with respect to reflection automorphisms, which establishes that the orbit structure of $\pop_T$ does not depend on the choice of Coxeter element $c$.  Thus, the Coxeter pop-tsack torsing operators defined with respect to different Coxeter elements are dynamically equivalent to each other, so none of our results about $\pop_T$ depend on the choice of $c$. This means that we have the freedom to choose any Coxeter element we want when we work with specific Coxeter groups, and it will often be helpful to use a specific choice that makes the resulting combinatorics easiest to describe.

\begin{proposition}  Let $\phi$ be a reflection automorphism of $W$.
For $w \in W$, we have
\begin{itemize}
\item $\phi(\pi_T(w,\phi^{-1}(c))) = \pi_T(\phi(w),c)$, and
\item $\phi(\pop_T(w,\phi^{-1}(c))) = \pop_T(\phi(w),c)$.
\end{itemize}
\label{prop:pop_props}
\end{proposition}

\begin{proof}
For the first point, observe that the map $w \mapsto \phi(w)$ is a lattice isomorphism from $\NC(W,c)$ to $\NC(W,\phi(c))$.  Furthermore, for each reflection $t$, we have $t \leq_T w$ if and only if $\phi(t) \leq_T \phi(w)$.  Therefore, \[\phi\left(\pi_T(w,\phi^{-1}(c))\right) = \phi\left(\bigvee_{t \leq_T w}^{\NC(W,\phi^{-1}(c))} t\right) = \bigvee_{t \leq_T w}^{\NC(W,c)} \phi(t) = \bigvee_{t \leq_T \phi(w)}^{\NC(W,c)} t = \pi_T(\phi(w),c).\]  

For the second point, we use the first point to check that 
\begin{align*}\pop_T(\phi(w),c) &= \phi(w) \pi_T(\phi(w),c)^{-1} = \phi(w) \left(\phi\left(\pi_T(w,\phi^{-1}(c))\right)\right)^{-1}\\ & = \phi\left(w\pi_T(w,\phi^{-1}( c))^{-1}\right) = \phi\left(\pop_T(w,\phi^{-1}(c))\right) .\qedhere\end{align*}
\end{proof}

As conjugation by $c$ is a special kind of reflection automorphism that restricts to a lattice isomorphism of $\NC(W,c)$, we can specialize the preceding result to conclude the following.

\begin{corollary} For $w \in W$, we have
\begin{itemize}
    \item  $\pi_T(w,c)^c = \pi_T(w^c,c)$, and
    \item $\pop_T(w,c)^c = \pop_T(w^c,c)$. 
\end{itemize}
\end{corollary}

Since reflection automorphisms act transitively on the set of Coxeter elements, we further obtain:
\begin{corollary}
\label{cor:orbit_structure}
The orbit structure of $\pop_T$ does not depend on the choice of Coxeter element.
\end{corollary}

\subsection{Folding and the Coxeter plane}
\label{sec:folding}

\begin{definition}

Let $(W,S)$ and $(\U,S')$ be Coxeter systems such that $W$ and $W'$ are finite and irreducible. Suppose there is a surjective map $\fold:S\to S'$ such that for each $s'\in S'$, the elements of $\fold^{-1}(s')$ all commute with each other. If there is an injective homomorphism $\unfold:\U\to W$ with $\unfold(s')=\prod_{s\in\fold^{-1}(s')}s$ for all $s'\in S'$, then we say that $W$ \defn{folds} to $W'$ and that $W'$ \defn{unfolds} into $W$. 
\end{definition}

\Cref{fig:folding} illustrates several examples of folding. If $W$ folds to $\U$ and $c'$ is a Coxeter element of $\U$, then $\unfold(c')$ is a Coxeter element of $W$. Hence, the Coxeter number is preserved by folding.

A useful example of folding goes through the Coxeter plane.  Any finite irreducible Coxeter group $W$ with Coxeter number $h$ folds to the dihedral group $I_2(h)$ of order $2h$, as we now explain.  Since the Coxeter-Dynkin diagram of $W$ is a tree (hence, a bipartite graph), the set $S$ of simple reflections of $W$ naturally separates into two sets $S_+$ and $S_-$, where the simple reflections within each set all commute with each other.  Writing $c_\pm$ for the product of the reflections in $S_\pm$ so that $c = c_+ c_-$ is a Coxeter element of $W$, $\langle c_+,c_-\rangle$ is the dihedral group $I_2(h)$.  Each of $c_+$ and $c_-$ acts as a reflection on the \defn{Coxeter plane}, a plane spanned by the real and imaginary parts of an eigenvector for $c$ with eigenvalue $e^{2\pi i/h}$.  \Cref{fig:arr_coxplane} illustrates this for the symmetric group $\mathfrak{S}_4$ with Coxeter plane the dihedral group $I_2(4)$.

\begin{figure}[htbp]
    \centering
    \bgroup
\def\arraystretch{4}
\begin{tabular}{ccc|c}
     $W$ & $\U$ & $h$ & Diagrammatic folding  \\ \hline
     $A_{2n-1}$ & $B_n$ & $2n$ & \raisebox{-0.5\height}{\begin{tikzpicture}[scale=.5]
       \filldraw[black] (0,.5) circle (2pt) -- (1,0) circle (2pt) -- (2,0) circle (2pt);
        \filldraw[black,dotted] (2,0) circle (2pt) -- (3,0) circle (2pt);
        \filldraw[black] (3,0) circle (2pt) -- (4,0) circle (2pt);
        \filldraw[black] (0,.5) circle (2pt) -- (1,1) circle (2pt) -- (2,1) circle (2pt);
        \filldraw[black,dotted] (2,1) circle (2pt) -- (3,1) circle (2pt);
        \filldraw[black] (3,1) circle (2pt) -- (4,1) circle (2pt);
        \filldraw[black,double] (0,-.5) -- (1,-.5);
        \filldraw[black] (0,-.5) circle (2pt);
        \filldraw[black] (1,-.5) circle (2pt);
        \filldraw[black] (1,-.5) circle (2pt) -- (2,-.5) circle (2pt);
        \filldraw[black,dotted] (2,-.5) circle (2pt) -- (3,-.5) circle (2pt);
        \filldraw[black] (3,-.5) circle (2pt) -- (4,-.5) circle (2pt);
      \end{tikzpicture}}\\ 
     $D_n$ & $B_{n-1}$ & $2n-2$ & \raisebox{-0.5\height}{\begin{tikzpicture}[scale=.5]
     \filldraw[black] (0,1) circle (2pt) -- (1,.5) circle (2pt);
       \filldraw[black] (0,0) circle (2pt) -- (1,.5) circle (2pt) -- (2,.5) circle (2pt);
        \filldraw[black,dotted] (2,.5) circle (2pt) -- (3,.5) circle (2pt);
        \filldraw[black] (3,.5) circle (2pt) -- (4,.5) circle (2pt);
        \filldraw[black,double] (0,-.5) -- (1,-.5);
        \filldraw[black] (0,-.5) circle (2pt);
        \filldraw[black] (1,-.5) circle (2pt);
        \filldraw[black] (1,-.5) circle (2pt) -- (2,-.5) circle (2pt);
        \filldraw[black,dotted] (2,-.5) circle (2pt) -- (3,-.5) circle (2pt);
        \filldraw[black] (3,-.5) circle (2pt) -- (4,-.5) circle (2pt);
      \end{tikzpicture}} \\ 
     $E_6$ & $F_4$ & $12$ & \raisebox{-0.5\height}{\begin{tikzpicture}[scale=.5]
       \filldraw[black] (-1,.5) circle(2pt) -- (0,.5) circle (2pt) -- (1,0) circle (2pt) -- (2,0) circle (2pt);
        \filldraw[black] (0,.5) circle (2pt) -- (1,1) circle (2pt) -- (2,1) circle (2pt);
        \filldraw[black,double] (0,-.5) -- (1,-.5);
        \filldraw[black] (-1,-.5) circle (2pt) -- (0,-.5) circle (2pt);
        \filldraw[black] (0,-.5) circle (2pt);
        \filldraw[black] (1,-.5) circle (2pt);
        \filldraw[black] (1,-.5) circle (2pt) -- (2,-.5) circle (2pt);
      \end{tikzpicture}}\\ 
     $E_8$ & $H_4$ & 30 & \raisebox{-0.5\height}{\begin{tikzpicture}[scale=.5]
       \filldraw[black] (0,1.5) circle(2pt) -- (1,1) circle (2pt) -- (2,1) circle (2pt) -- (3,1) circle (2pt);
        \filldraw[black] (0,.5) circle (2pt) -- (1,0) circle (2pt) -- (2,0) circle (2pt) -- (3,0) circle (2pt);
        \filldraw[black] (0,.5) -- (1,1);
        \filldraw[black] (0,-.5) circle (2pt) -- (1,-.5) circle (2pt) -- (2,-.5) circle (2pt) -- (3,-.5) circle (2pt);
        \node at (.5,-1) {\scriptsize 5};
      \end{tikzpicture}}\\
\end{tabular}
\egroup
    \caption{Examples of folding of finite Coxeter groups.}
    \label{fig:folding}
\end{figure}

In this subsection, we write $T$ and $T'$ for the sets of reflections in $W$ and $W'$, respectively. 

\begin{lemma}
Let $\U$ be a finite irreducible Coxeter group that unfolds into $W$. Let $c'$ be a bipartite Coxeter element of $\U$, and let $c=\unfold(c')$ be the corresponding Coxeter element of $W$. Then $\unfold(\NC(\U),c')$ is a sublattice of $\NC(W,c)$.
\label{lem:lattice_embedding}
\end{lemma}
\begin{proof}
Suppose that $w'\in W$ and $w' \leq_{T'} c'$ so that $\uu$ can be found as a prefix of $c'$. Then $\unfold(\uu)$ can be found as a prefix of $c$, so $\unfold(\uu) \leq_T c$.  By the same reasoning (replacing now $c'$ by any element of $\NC(\U,c')$), we have that $\unfold(\NC(\U,c'))$ is a subposet of $\NC(W,c)$.

We now check that $\unfold(\NC(\U,c'))$ is a sublattice of $\NC(W,c)$ (that is, that $\unfold(\NC(\U,c'))$ is closed under the meet and join operations of $\NC(W,c)$).  We know by~\cite{brady2008non} that a noncrossing partition $w$ is characterized by the set of reflections that lie below it: $w = \bigvee_{t \leq_T w} t$. Since $w \wedge v = \bigvee_{t : t \leq_T w \text{ and } t \leq_T v} t$, it follows that $\unfold(\NC(\U))$ is a meet-sublattice. Indeed, for fixed $t'\in T'$ and $r \in \fold^{-1}(t')$, we have $r \leq \unfold(w') \wedge \unfold(v')$ if and only if $r \leq_T \unfold(w')$ and $r \leq_T \unfold(v')$. This occurs if and only if $t' \leq_{T'} w'$ and $t' \leq_{T'} v'$, which occurs if and only if $t' \leq_{T'} w' \wedge v'$, which occurs if and only if $r \leq_T \unfold(w' \wedge v')$. 

The Kreweras complement maps $K:\NC(W,c)\to \NC(W,c)$ and $K':\NC(W',c')\to \NC(W',c')$, defined by $K(w)= cw^{-1}$ and $K'(w')=c'(w')^{-1}$, are lattice anti-isomorphisms \cite{armstrong2009generalized}. We have $K\circ\unfold=\unfold\circ K'$ because $\unfold$ is a homomorphism that sends $c'$ to $c$. This shows that $K$ preserves the set $\unfold(\NC(\U,c'))$, so $\unfold(\NC(\U,c'))$ is also a join-sublattice of $\NC(W,c)$.
\end{proof}

\begin{theorem}
\label{thm:pop_embedding}
If a finite irreducible Coxeter group $W'$ unfolds into $W$, then for $w' \in W'$,
\[\unfold(\pop_{T'}^k(w')) = \pop_T^k(\unfold(w')).\]
\end{theorem}

\begin{proof}
It suffices to prove the case when $k=1$ since the general result then follows by induction. Choose $w'\in \U$. If $t'\in T'$ is such that $t'\leq_{T'} w'$, then $\unfold(t')\leq_{T'}\unfold(w')$. By the lattice embedding of~\Cref{lem:lattice_embedding}, we have $\pi_T(\U)=\unfold(\bigvee_{t'\leq_\U \uu} t') = \bigvee_{t'\leq_\U \uu} \unfold(t')=\pi_T(\unfold(\uu))$.
Because $\unfold$ is a group homomorphism, we have \[\unfold(\pop_{T'}(\uu))=\unfold(\uu\pi_{T'}(\uu)^{-1})=\unfold(\uu)\unfold(\pi_{T'}(\uu))^{-1}=\unfold(\uu)\pi_T(\unfold(\uu))^{-1}\] \[=\pop_T(\unfold(\uu)).\qedhere\]
\end{proof}

We can now prove the part of~\Cref{thm:cinverseorbit} stating that $c^{-1}$ requires $h-1$ iterations of $\pop_T$ to reach the identity.
\begin{corollary}\label{cor:cinverseorbit}
We have $\pop_T^{h-1}(c^{-1})=e$. If $W$ is not the trivial group, then $\pop_T^{h-2}(c^{-1})\neq e$.
\end{corollary}
\begin{proof}
If $W$ is trivial, then so is the result, so we may assume $|W|\geq 2$. Let $W'=I_2(h)$, and let $c'$ be a Coxeter element of $W'$. We have seen that $W$ folds to $W'$ via the Coxeter plane and that $\unfold(c')$ is a Coxeter element of $W$. By \Cref{prop:pop_props}, we may assume $c=\unfold(c')$. As discussed in \Cref{sec:dihedral}, we have $\pop_{T'}^{h-2}(c^{-1})\neq\pop_{T'}(c^{-1})=e$. The result now follows from \Cref{thm:pop_embedding} and the fact that $\unfold:W'\to W$ is injective.
\end{proof}

\subsection{The inverse Coxeter element orbit}
\label{sec:inverse_c_orbit}

In this section, we prove the following theorem regarding the forward orbit of $c^{-1}$ under $\pop_T(\cdot,c)$.
\begin{theorem}\label{thm:cinverseorbit}
The forward orbit of $c^{-1}$ under $\pop_T(\cdot,c)$ is \[O_{\pop_T}(c^{-1})=\{c^{-1}, c^{-2},\ldots, c^{-(h-1)}, e\}.\] The element $c^{-1}$ has no preimages under $\pop_T$. Moreover, if $2\leq i\leq h-1$, then the only preimage of $c^{-i}$ under $\pop_T$ is $c^{-(i-1)}$. 
\end{theorem}

Using the fact that $W$ folds to $I_2(h)$ via the Coxeter plane, we proved in \Cref{cor:cinverseorbit} that $c^{-1}$ requires $(h-1)$ iterations of $\pop_T$ to reach the identity. Let us provide an alternative proof of this fact. We require a lemma, which states that the join in the noncrossing partition lattice of any orbit of a reflection under the conjugation action of $c$ is just $c$. 

\begin{lemma}\label{Lem:joinofcorbit}
If $t\in T$ is a reflection, then $c=\bigvee\limits_{k=0}^{h-1} t^{c^{k}}$.
\end{lemma}

\begin{proof}
Since $W$ folds to $I_2(h)$ via the Coxeter plane, we can consider the Coxeter element $\unfold(c')$ of $W$, where $c'$ is a Coxeter element of $I_2(h)$. Each Coxeter element of $W$ can be obtained by applying a reflection automorphism to $\unfold(c')$, so it suffices to prove the lemma when $c=\unfold(c')$.

Conjugation by $c$ induces an automorphism of $\NC(W,c)$; therefore, if a set $O \subseteq \NC(W,c)$ is invariant under conjugation by $c$, then the element $x=\bigvee_{w \in O} w$ is invariant under conjugation by $c$:
	\[x^c = \left(\bigvee_{w \in O} w\right)^c = \bigvee_{w \in O} w^c = \bigvee_{w \in O} w = x.\]
It is a classical result that the centralizer of $c$ is the cyclic group generated by $c$ (see \cite[Proposition~30]{Carter}), so we must have $x=c^i$ for some integer $i$.  Therefore, we are left to prove that the only powers of $c$ in $\NC(W,c)$ are $c$ and $e$. This will follow if we can show that the only powers of $c'$ in $\NC(I_2(h),c')$ are $c'$ and $e$.  But this is clear since the absolute length of any power of $c'$ (a rotation) is $0$ or $2$ (see also~\cite[Lemma 12.2]{bessis2015finite}).
\end{proof}

\begin{proposition}
For each integer $1\leq i\leq h-1$, we have $\pop_T(c^{-i})=c^{-(i+1)}$. Thus, the forward orbit of $c^{-1}$ under $\pop_T$ is $\{c^{-1},c^{-2},\ldots,c^{-(h-1)},e\}$, which has size $h$.
\label{prop:forward_cinv}
\end{proposition}

\begin{proof}
Fix $1\leq i\leq h-1$. We claim that if $t\in T$ is such that $t \leq_T c^{-i}$, then $t^c \leq_T c^{-i}$.  According to \eqref{EqBradyWatt}, we have $\mov(t) \subseteq \mov(c^{-i})$.  Take $\alpha_t$ to be a vector orthogonal to the hyperplane $H_t$ so that $\mov(t) = \spn\{\alpha_t\}$.  Then $\alpha_t\in\mov(c^{-i})$, so there exists a vector $\beta\in V$ such that $(c^{-i}-1)\cdot\beta = \alpha_t$. Applying $c$ to each side of this equation shows that \[(c^{-i}-1)\cdot(c\cdot\beta)=c\cdot((c^{-i}-1)\cdot\beta)=c\cdot\alpha_t.\] Thus, $\spn\{c\cdot\alpha_t\}\subseteq\mov(c^{-i})$. Since $c\cdot H_t=H_{t^c}$, we know that $c\cdot\alpha_t$ is orthogonal to $H_{t^c}$. Thus, $\mov(t^c)=\spn\{c\cdot\alpha_t\}$. Appealing to \eqref{EqBradyWatt} again, we find that $t^c\leq_T c^{-i}$. 

Note that there must exist some $t\in T$ with $t \leq_T c^{-i}$ since $c^{-i} \neq e$. It follows from the preceding paragraph that $t^{c^k}\leq_T c^{-i}$ for all $0\leq k\leq h-1$. Therefore, we can use \Cref{Lem:joinofcorbit} to see that \[c=\bigvee\limits_{k=0}^{h-1} t^{c^{k}}\leq_T\bigvee_{t'\leq_T c^{-i}}t'=\pi_T(c^{-i}).\] Consequently, $\pi_T(c^{-i})=c$, so $\pop_T(c^{-i})=c^{-i}c^{-1}=c^{-(i+1)}$. 
\end{proof}

\begin{corollary}
The element $c^{-1}$ has no preimages under $\pop_T$.  Moreover, if $2 \leq i \leq h-1$, then the only preimage of $c^{-1}$ under $\pop_T$ is $c^{-(i-1)}$.
\end{corollary}

\begin{proof}
  By \Cref{prop:preimages} and \Cref{prop:forward_cinv}, the only possible preimage of $c^{-i}$ under $\pop_T$ is $c^{-(i-1)}$.  Furthermore, $c^{-(i-1)}$ is actually a preimage of $c^{-i}$ unless $i=1$.
\end{proof}

This completes the proof of \Cref{thm:cinverseorbit}.

\section{Types $A$ and $B$}
\label{sec:typeab}
In this section, we prove the following theorem using type-by-type analysis of the dynamics of $\pop_T$ on Coxeter groups of types~$A$ and $B$.  (The dihedral case was addressed in~\Cref{sec:dihedral}, and $H_3$ was checked by computer).
\begin{theorem}\label{thm:main1}
Let $W$ be a Coxeter group of coincidental type with a fixed Coxeter element $c$ and Coxeter number $h$. Then 
\begin{itemize}
    \item $\pop_T^{h-1}(w)=e$ for all $w\in W$,
    \item  $\max\limits _{w\in W}\left|O_{\pop_T}(w)\right|=h$, and
    \item the only forward orbit of $\pop_T$ of size $h$ is $O_{\pop_T}(c^{-1})$. 
\end{itemize}
\end{theorem}

\subsection{Type $A$}
\label{sec:typea}
Recall that the Coxeter group of type~$A_{n-1}$ is the symmetric group $\mathfrak S_n$, the group of permutations of $[n]$. We write permutations in disjoint cycle notation. Throughout this section, we fix the Coxeter element $c=(12\cdots n)$. 

There is a simple combinatorial way to describe the noncrossing partitions in $\NC(\mathfrak S_n,c)$. Begin by labeling $n$ equally-spaced points on a circle with the numbers $1,\ldots,n$ in clockwise order. Now choose a set partition $\rho$ of $[n]$. For each block of $\rho$, draw the convex hull of the points on the circle labeled with the elements of that block. We say $\rho$ is a \defn{noncrossing set partition} if none of the convex hulls of different blocks intersect each other. Suppose $\rho$ is a noncrossing set partition. For each block of $\rho$, form a cycle by ordering the elements of the block cyclically in clockwise order around the circle. The product of these cycles is a noncrossing partition in $\NC(\mathfrak S_n,c)$, and every noncrossing partition in $\NC(\mathfrak S_n,c)$ arises uniquely in this way. See \Cref{FigTsack1}.

\begin{figure}[htbp]
\begin{center} 
\includegraphics[height=4cm]{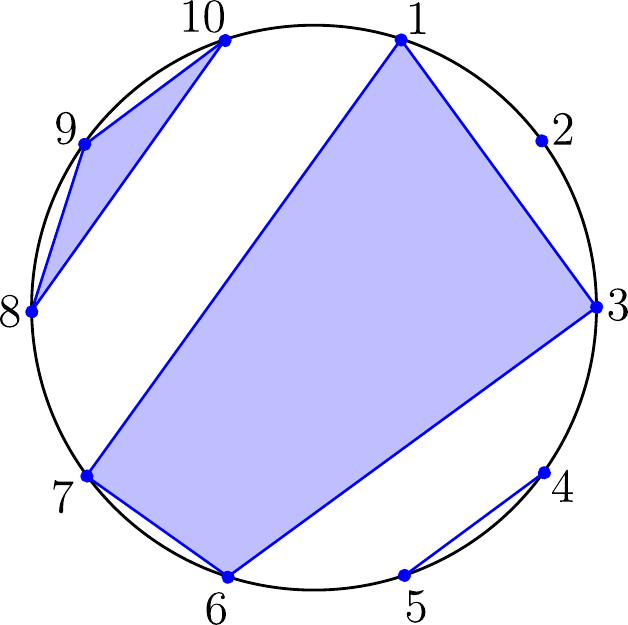}
\end{center}
\caption{A noncrossing set partition. The corresponding noncrossing partition in $\NC(\mathfrak S_n,c)$ is $(1\,\,3\,\,6\,\,7)(2)(4\,\,5)(8\,\,9\,\,10)$.}\label{FigTsack1}
\end{figure}

We will find it convenient to refer to a cycle of a permutation and to the set of entries in the cycle interchangeably. For example, we sometimes write $i\in\mathcal C$ to mean that $i$ is an entry in a cycle $\mathcal C$. Let $\cyc_{>1}(w)$ denote the number of non-singleton cycles of a permutation $w$. 

The partial order on $\NC(\mathfrak S_n,c)$, which is the restriction of the absolute order on $\mathfrak S_n$ to $[e,c]$, is easy to describe combinatorially because it corresponds to the reverse refinement order on noncrossing set partitions. In other words, if $u,v\in\NC(\mathfrak S_n,c)$, then $u\leq_T v$ if and only if every cycle in $u$ is contained (as a set) in a cycle in $v$. The noncrossing projection of a permutation $w\in \mathfrak S_n$ is the smallest noncrossing partition $v\in\NC(\mathfrak S_n,c)$ such that every cycle in $w$ is contained in a cycle in $v$.

An \defn{antiexceedance} of a permutation $w \in \mathfrak{S}_n$ is an element $i\in[n]$ such that $i<w^{-1}(i)$. Write $\Aexc(w)$ for the set of antiexceedances of $w$, and let $\aexc=\lvert\Aexc(w)\rvert$. 

\begin{theorem}\label{thm:typeAantiexc} The following properties relating antiexceedances and $\pop_T$ hold:
\begin{itemize}
  \item The element $c^{-1}$ has $h-1=n-1$ antiexceedances. 
  \item Every element of $\mathfrak S_n$ other than $c^{-1}$ has at most $n-2$ antiexceedances. 
  \item For every $w\in \mathfrak S_n$, we have $\aexc(\pop_T(w)) = \aexc(w)-\cyc_{>1}(\pi_T(w))$.
\end{itemize}
\end{theorem}

\begin{proof}
The set of antiexceedances of a permutation $w\in \mathfrak S_n$ is a subset of $[n-1]$. It is straightforward to see that this set of antiexceedances is equal to $[n-1]$ if and only if $w=c^{-1}=(n\cdots 21)$. This proves the first two bulleted items. Now fix $w\in \mathfrak S_n$, and let $v=\pi_T(w)$ so that $\pop_T(w)=wv^{-1}$.

Suppose $i\in[n]$ is not an antiexceedance of $w$. This means that $i\geq w^{-1}(i)$. Let $\mathcal C$ be the cycle of $v$ that contains $i$. Since every cycle in $w$ is contained in a cycle in $v$, we know that $\mathcal C$ also contains $w^{-1}(i)$. The cycle $\mathcal C$ is cyclically ordered clockwise, so $v(w^{-1}(i))$, which is the entry immediately following $w^{-1}(i)$ in $\mathcal C$, is less than or equal to $i$. Because $v(w^{-1}(i))=(\pop_T(w))^{-1}(i)$, this shows that $i$ is not an antiexceedance of $\pop_T(w)$. As $i$ was arbitrary, this proves that $\Aexc(\pop_T(w))\subseteq\Aexc(w)$. 

We claim $j\in\Aexc(w)\setminus\Aexc(\pop_T(w))$ if and only if $v^{-1}(j)$ is the largest entry in a non-singleton cycle. Note that this claim immediately implies the third bulleted item in the statement of the theorem. To prove this claim, suppose first that $j\in\Aexc(w)\setminus\Aexc(\pop_T(w))$. The cycle of $w$ containing $j$ and $w^{-1}(j)$ is contained in a cycle $\mathcal C$ of $v$; note that $\mathcal C$ is not a singleton cycle because $w^{-1}(j)\neq j$. If $w^{-1}(j)$ is not the largest entry in $\mathcal C$, then $v(w^{-1}(j))$, which is the entry immediately following $w^{-1}(j)$ in $\mathcal C$, is larger than $w^{-1}(j)$. However, this implies that $\pop_T(w)^{-1}(j)=v(w^{-1}(j))>w^{-1}(j)>j$, contradicting the fact that $j\not\in\Aexc(\pop_T(w))$. This proves that $w^{-1}(j)$ must be the largest entry in a non-singleton cycle in $v$. 

To prove the reverse containment, suppose $w^{-1}(j)$ is the largest entry in a non-singleton cycle $\mathcal C$ in $v$. Note that $j\in\mathcal C$. If $j$ were a fixed point of $w$, then it would also be a fixed point of $v$, contradicting the fact that $\mathcal C$ is not a singleton cycle. Therefore, $w^{-1}(j)>j$, so $j\in\Aexc(w)$. Observe that $v(w^{-1}(j))$, which is the entry immediately following $w^{-1}(j)$ in $\mathcal C$, is the smallest entry in $\mathcal C$. Therefore, $v(w^{-1}(j))\leq j$. This proves that $j\not\in\Aexc(\pop_T(w))$, as desired. 
\end{proof}


\begin{example}
Consider the permutation $w=(1\,\,2\,\,4\,\,6\,\,5)(3)(7\,\,9)(8\,\,10)$ in $\mathfrak S_{10}$. This permutation is represented pictorially on the left of \Cref{FigTsack2}; its noncrossing projection $\pi_T(w)$ is $(1\,\,2\,\,4\,\,5\,\,6)(3)(7\,\,8\,\,9\,\,10)$, which is drawn in the middle of \Cref{FigTsack2}. Finally, $\pop_T(w)=w\pi_T(w)^{-1}$ appears on the right side of the figure. The $4$ anti-exceedances of $w$ (shown in red in the figure) are $1,5,7,8$. Since $\pi_T(w)$ has $2$ non-singleton cycles, the permutation $\pop_T(w)$ has $4-2=2$ anti-exceedances---namely, $1$ and $7$. 
\end{example}

\begin{figure}[htbp]
\begin{center} 
\includegraphics[height=4cm]{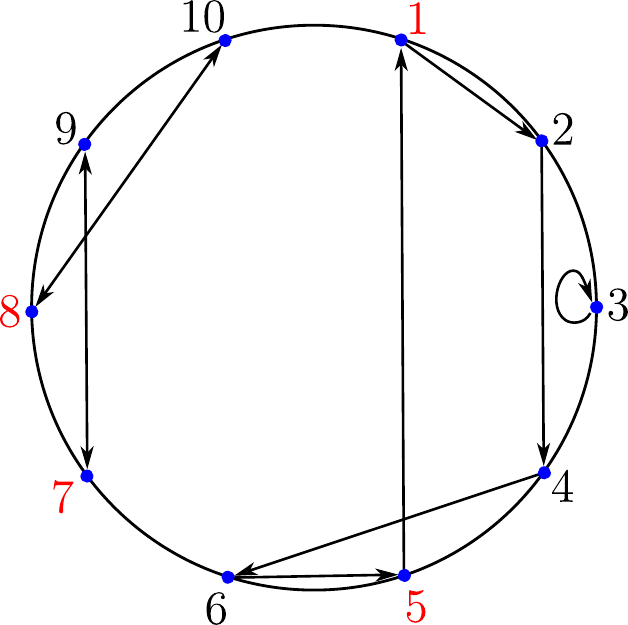}\qquad\qquad\includegraphics[height=4cm]{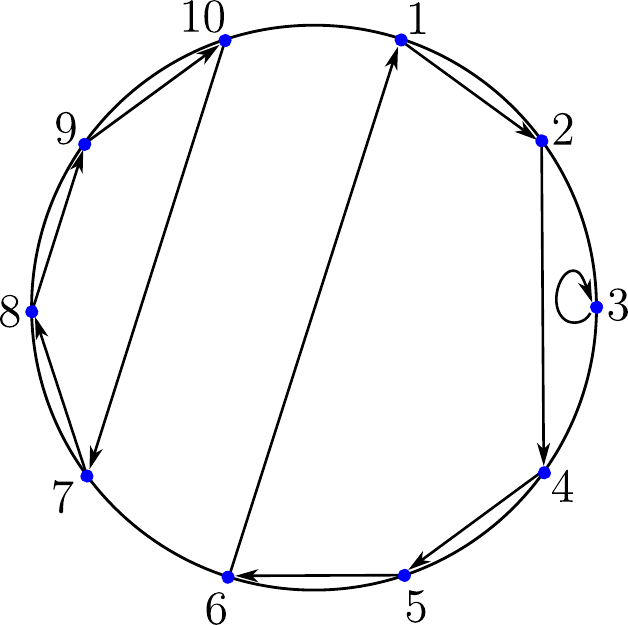}\qquad\qquad\includegraphics[height=4cm]{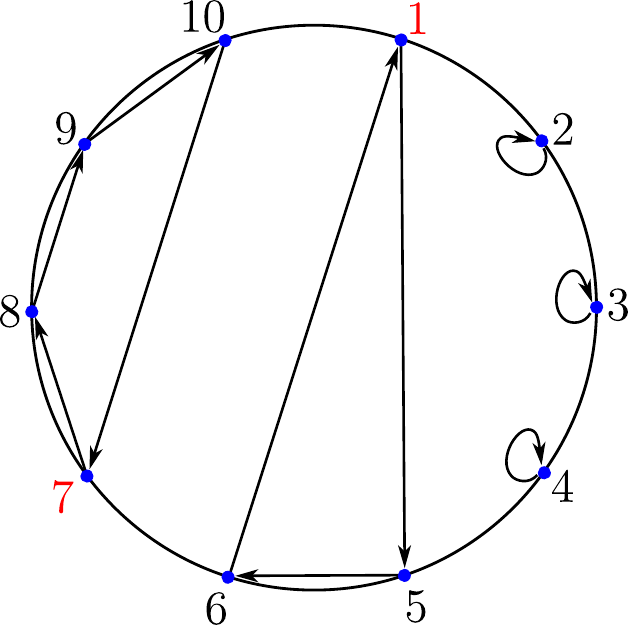}
\end{center}
\caption{A permutation (left), its noncrossing projection (middle), and its image under $\pop_T$ (right).}\label{FigTsack2}
\end{figure}

The Coxeter number of $\mathfrak S_n$ is $n$, so the following corollary completes the proof of \Cref{thm:main1} in type~$A$. 

\begin{corollary}
Every forward orbit of the map $\pop_T:\mathfrak S_n\to \mathfrak S_n$ has size at most $n$, and the only forward orbit that has size equal to $n$ is $O_{\pop_T}(c^{-1})$. 
\end{corollary}

\begin{proof}
We know that $O_{\pop_T}(c^{-1})$ has size $n$ by \Cref{prop:forward_cinv}. The only element of $\mathfrak S_n$ whose noncrossing projection has no non-singleton cycles is the identity element, so it follows from the third bulleted item in \Cref{thm:typeAantiexc} that $\aexc(\pop_T(w))\leq\aexc(w)-1$ for all $w\in \mathfrak S_n\setminus\{e\}$. The only element of $\mathfrak S_n$ with $0$ antiexceedances is the identity element, so it follows from the second bulleted item in \Cref{thm:typeAantiexc} that $\pop_T^{n-2}(w)=e$ for all $w\in \mathfrak S_n\setminus\{c^{-1}\}$. 
\end{proof}

\begin{conjecture}
\label{conj:A}
In type $A_{n-1}$, the number of permutations that require exactly $n-2$ or $n-1$ iterations of $\pop_T$ to reach the identity is $2^n-\binom{n}{2}$.
\end{conjecture}

\subsection{Type $B$}
\label{sec:typeb}

The previous section established \Cref{thm:main1} in type~$A$. The group $A_{2n-1}$ folds to $B_n$, so it follows immediately from \Cref{thm:pop_embedding} that \Cref{thm:main1} must also hold in type~$B$. In this section, we will verify \Cref{thm:main1} in type~$B$ directly using the combinatorics of the hyperoctahedral groups; this will also serve as a useful warm-up for our analysis of type~$D$ in the next section.


For the sake of convenience, we replace negative signs with bars, writing $\overline i$ instead of $-i$. Let $\pm[n]=\{\overline n,\ldots,\overline 1,1,\ldots,n\}$. The \defn{hyperoctahedral group} $B_n$ is the group of permutations $w$ of the set $\pm[n]$ such that $w(\overline i)=\overline{w(i)}$ for all $i\in[n]$. Throughout this section, we fix the Coxeter element $c=(\overline 1\,\,\overline 2\,\cdots \overline n\,\,1\,\,2\cdots n)$ of $B_n$. 

Label $2n$ equally-spaced points on a circle with the numbers $\overline 1,\overline 2,\ldots,\overline n,1,2,$ $\ldots,n$ in clockwise order. Now choose a set partition $\rho$ of $\pm[n]$. For each block of $\rho$, draw the convex hull of the points on the circle labeled with the elements of that block. We say $\rho$ is a \defn{type-$B_n$ noncrossing set partition} if the resulting diagram is invariant under a $180^\circ$ rotation and none of the convex hulls of different blocks intersect each other. Suppose $\rho$ is a type-$B_n$ noncrossing set partition. For each block of $\rho$, form a cycle by ordering the elements of the block cyclically in clockwise order around the circle. The product of these cycles is a noncrossing partition in $\NC(B_n,c)$, and every noncrossing partition in $\NC(B_n,c)$ arises uniquely in this way. 

\begin{figure}[htbp]
\begin{center} 
\includegraphics[height=4cm]{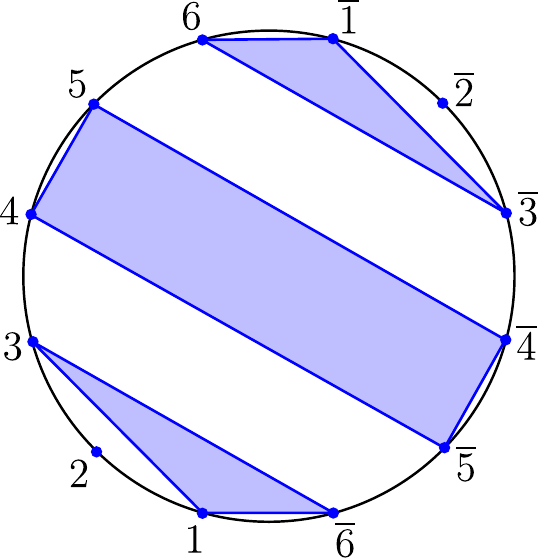}
\end{center}
\caption{A type-$B_6$ noncrossing set partition. The corresponding noncrossing partition in $\NC(B_n,c)$ is $(\overline 1\,\,\overline 3\,\,6)(1\,\,3\,\,\overline 6)(\overline 2)(2)(\overline 4\,\,\overline 5\,\,4\,\,5)$.}\label{FigTsack3}
\end{figure}

As in the previous section, we refer to a cycle of an element of $B_n$ and to the set of entries in the cycle interchangeably. Let $\cyc_{>1}(w)$ be the number of non-singleton cycles in $w$. 

The partial order on $\NC(B_n,c)$ corresponds to the reverse refinement order on type-$B_n$ noncrossing set partitions. That is, if $u,v\in\NC(B_n,c)$, then $u\leq_T v$ if and only if every cycle in $u$ is contained in a cycle in $v$. The noncrossing projection of $w\in B_n$ is the smallest noncrossing partition $v\in\NC(B_n,c)$ such that every cycle in $w$ is contained in a cycle in $v$.

Define a total order $\prec$ on the set $\pm[n]$ by 
\[\overline 1 \prec \overline 2 \prec \cdots \prec \overline n \prec 1 \prec 2 \prec \cdots \prec n.\] An \defn{antiexceedance} of $w \in B_n$ is an element $i\in[n]$ such that $i\prec w^{-1}(i)$. Write $\Aexc(w)$ for the set of antiexceedances of $w$, and let $\aexc=\lvert\Aexc(w)\rvert$.

\begin{theorem}\label{thm:typeBantiexc}
The following properties relating antiexceedances and $\pop_T$ hold:
\begin{itemize}
  \item The element $c^{-1}$ has $h-1=2n-1$ antiexceedances. 
  \item Every element of $B_n$ other than $c^{-1}$ has at most $2n-2$ antiexceedances. 
  \item For every $w\in B_n$, we have $\aexc(\pop_T(w)) = \aexc(w)-\cyc_{>1}(w)$.  
\end{itemize}
\end{theorem}

\begin{proof}
The proof is virtually identical to that of \Cref{thm:typeAantiexc}. As before, the first two bulleted items are easy to check directly. For the third bulleted item, fix $w\in B_n$, and let $v=\pi_T(w)$. Mimicking the proof of \Cref{thm:typeAantiexc}, one can show that $\Aexc(\pop_T(w))\subseteq\Aexc(w)$. Furthermore, we have $j\in\Aexc(w)\setminus\Aexc(\pop_T(w))$ if and only if $v^{-1}(j)$ is the largest entry in a non-singleton cycle, where the word ``largest'' refers to the total order $\prec$. 
\end{proof}

The Coxeter number of $B_n$ is $2n$, so the following corollary completes the proof of \Cref{thm:main1} in type~$B$. 

\begin{corollary}
Every forward orbit of the map $\pop_T:B_n\to B_n$ has size at most $2n$, and the only forward orbit that has size equal to $2n$ is $O_{\pop_T}(c^{-1})$. 
\end{corollary}

\begin{proof}
We know that $O_{\pop_T}(c^{-1})$ has size $2n$ by \Cref{prop:forward_cinv}. The only element of $B_n$ whose noncrossing projection has no non-singleton cycles is the identity element, so it follows from the third bulleted item in \Cref{thm:typeBantiexc} that $\aexc(\pop_T(w))\leq\aexc(w)-1$ for all $w\in \mathfrak S_n\setminus\{e\}$. The only element of $B_n$ with $0$ antiexceedances is the identity element, so it follows from the second bulleted item in \Cref{thm:typeBantiexc} that $\pop_T^{2n-2}(w)=e$ for all $w\in B_n\setminus\{c^{-1}\}$. 
\end{proof}

\begin{conjecture}
\label{conj:B}
The number of elements of $B_n$ that require exactly $2n-2$ or $2n-1$ iterations of $\pop_T$ to reach the identity is $2^n-n$.
\end{conjecture}


\section{Type $D$}
\label{sec:typed}
In this section, we prove the following theorem.

\begin{theorem}\label{thm:main2}
Let $W$ be a Coxeter group of type $D_n$. Then 
\begin{itemize}
    \item $\pop_T^{2n-3}(w)=e$ for all $w\in W$, and
    \item  $\max\limits _{w\in W}\left|O_{\pop_T}(w)\right|=2n-2$.
\end{itemize}
\end{theorem}

As in the previous section, let $\pm[n]=\{\overline n,\ldots,\overline 1,1,\ldots,n\}$ (with bars replacing negative signs). The Coxeter group $D_n$ is the group of permutations $w:\pm[n]\to\pm[n]$ such that $w(\overline i)=\overline{w(i)}$ for all $i\in\pm[n]$ and such that $\abs{\{i\in[n]:w(i)<0\}}$ is even. The reflections in $D_n$ are the elements of the form $(i\,\,j)(\overline i\,\,\overline j)$, where $i,j\in\pm[n]$ and $i\neq\overline j$. Throughout this section, we fix the Coxeter element 
$c=(\overline 1\,\,\overline 2\,\cdots\overline{n-1}\,\,1\,\,2\cdots n-1)(\overline n\,\,n)$. Given a cycle $\mathcal C=(i_1i_2\cdots i_k)$ in an element $w\in D_n$, let $\overline{\mathcal C}$ denote the cycle $(\overline{i_1}\,\,\overline{i_2}\cdots \overline{i_k})$, and note that $\overline{\mathcal C}$ is also a cycle in $w$. We say the cycle $\mathcal C$ is \defn{balanced} if $\mathcal C=\overline{\mathcal C}$. It follows from the definition of $D_n$ that $\mathcal C$ is balanced if and only if there exists an entry $k$ such that $k$ and $\overline k$ are both in $\mathcal C$. A cycle that is not balanced is \defn{unbalanced}. Each element of $D_n$ has an even number of balanced cycles. 

Every element $w\in D_n$ can be written uniquely as \[w=\mathcal C_1\overline{\mathcal C_1}\cdots\mathcal C_r\overline{\mathcal C_r}\mathcal D_1\cdots\mathcal D_s,\] where $\mathcal C_1,\ldots,\mathcal C_r$ are unbalanced cycles, $\mathcal D_1,\ldots,\mathcal D_s$ are balanced cycles, and the cycles appearing in the factorization are pairwise disjoint. The reflection length of $w$ can be computed as \[\ell_T(w)=\sum_{i=1}^r(k_i-1)+\sum_{j=1}^sm_j,\] where $k_i$ is the number of entries in $\mathcal C_i$ (and, therefore, also $\overline{\mathcal C_i}$) and $m_j$ is the number of entries in $\mathcal D_j$ (see \cite[Section~3]{BradyWatt}). The following useful lemma is a consequence of the computations in \cite[Example~3.6]{BradyWatt} (see also \cite[Section~3]{AthanasiadisReiner}).

\begin{lemma}[\cite{BradyWatt}]\label{Lem1}
Let $w\in D_n$, and consider $i,j\in\pm[n]$ with $i\not\in\{\overline j,j\}$. We have $(i\,\,j)(\overline i\,\,\overline j)\leq_T w$ if and only if $i$ and $j$ are in the same cycle of $w$ or $i$ and $j$ are in different balanced cycles of $w$. 
\end{lemma}

Our combinatorial description of noncrossing partitions in type~$D$ is based on the work of Athanasiadis and Reiner \cite{AthanasiadisReiner}. Begin by labeling $2n{-}2$ equally-spaced points on a circle with the numbers $\overline 1,\overline 2,\ldots,$ $\overline{n-1},1,2,\ldots,n-1$ in clockwise order. Now place a single extra point at the center of the circle, and label it with both the number $n$ and the number $\overline n$. Consider a set partition $\rho$ of $\pm[n]$; we associate each block of $\rho$ with the set of points labeled by the elements of that block. We say $\rho$ is a \defn{type-$D_n$ noncrossing set partition} if it satisfies the following conditions: 
\begin{enumerate}
\item For every block $B$ of $\rho$, the set $\overline B=\{\overline i:i\in B\}$ is also a block of $\rho$.
\item If $B$ and $B'$ are distinct blocks of $\rho$, then the relative interior of the convex hull of $B$ does not contain a point of the convex hull of $B'$. 
\item The set $\{\overline n,n\}$ is not a block of $\rho$. 
\end{enumerate}
We associate each type-$D_n$ noncrossing set partition $\rho$ with the diagram showing the labeled points on the circle and at the center of the circle together with the convex hulls of the blocks of $\rho$. A \defn{zero block} of $\rho$ is a block $Z$ of $\rho$ such that $Z=\overline Z$. Note that $\rho$ has at most $1$ zero block; moreover, if $\rho$ does have a zero block $Z$, then $Z$ must properly contain the set $\{\overline n,n\}$. 

Suppose $\rho$ is a type-$D_n$ noncrossing set partition. For each block $B$ of $\rho$ that is not a zero block, form a cycle by reading the elements of $B$ in clockwise order around the boundary of the convex hull of $B$. If $\rho$ has a zero block $Z$, then form the cycle $(\overline n\,\, n)$, and form another cycle by reading the elements of $Z\setminus\{\overline n,n\}$ in clockwise order around the boundary of the convex hull of $Z$. Multiplying all of these cycles together forms the disjoint cycle decomposition of a permutation $g(\rho)$ of $\pm[n]$. Athanasiadis and Reiner \cite{AthanasiadisReiner} proved that $g(\rho)$ is actually a noncrossing partition in $\NC(D_n,c)$ and that the map $g$ provides a bijection from the set of type-$D_n$ noncrossing set partitions to the set $\NC(D_n,c)$.

\begin{figure}[htbp]
\begin{center} 
\includegraphics[height=4cm]{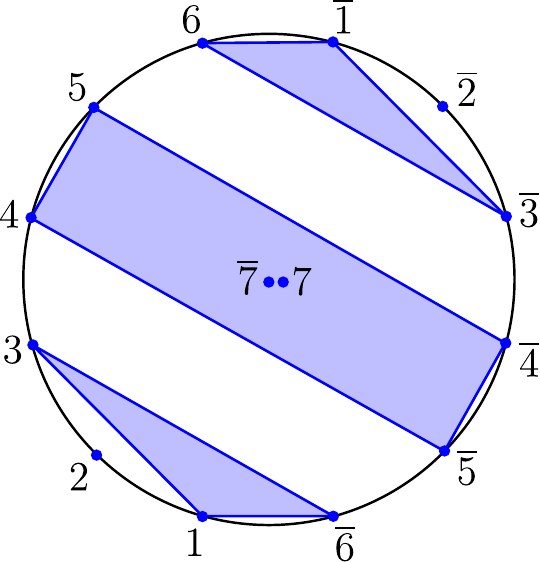}\qquad\qquad\includegraphics[height=4cm]{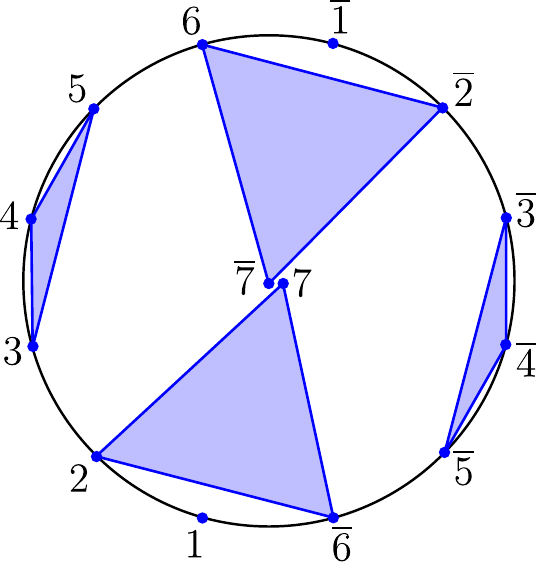}
\end{center}
\caption{Two type-$D_7$ noncrossing set partitions. Technically, $\overline 7$ and $7$ should both label a single point at the center of the circle, but we have drawn two points very close to each other in order to distinguish between set partitions whose diagrams would otherwise look identical. The type-$D_7$ noncrossing set partition on the left corresponds to the noncrossing partition $(\overline 1\,\,\overline 3\,\,6)(1\,\,3\,\,\overline 6)(\overline 2)(2)(\overline 4\,\,\overline 5\,\,4\,\,5)(\overline 7\,\,7)$, while the one on the right corresponds to $(\overline 1)(1)(\overline 2\,\,\overline 7\,\,6)(2\,\,7\,\,\overline 6)(\overline 3\,\,\overline 4\,\,\overline 5)(3\,\,4\,\,5)$.}\label{FigTsack4}
\end{figure}

\begin{lemma}\label{Lem3}
Suppose $v\in \NC(D_n,c)$ and $k\in\pm[n-1]$ are such that $v(k)\in\{\overline n,n\}$. Let $Y$ be the set consisting of the next $n-1$ numbers after $k$ in clockwise order around the circle. Then none of the elements of $Y$ lie in the same cycle as $k$ in $v$.  
\end{lemma}

\begin{proof}
Let $\mathcal C$ be the cycle in $v$ containing $k$. We can write $v=g(\rho)$ for some type-$D_n$ noncrossing set partition $\rho$. The entries in $\mathcal C$ form a block $B$ of $\rho$ that is not a zero block. By the definition of $g$, the cycle $\mathcal C$ is formed by reading the elements of $B$ in clockwise order around the boundary of the convex hull of $B$. This immediately implies that $\mathcal C\cap Y=\emptyset$.  
\end{proof}

\begin{lemma}\label{Lem2}
Let $u\in D_n$, and let $v=\pi_T(u)$. Consider distinct $i,j\in\pm[n-1]$. If $i$ and $j$ are in the same cycle of $u$, then they are in the same cycle of $v$.   
\end{lemma}

\begin{proof}
Suppose first that $i\neq \overline j$. It follows from \Cref{Lem1} that $(i\,\,j)(\overline i\,\,\overline j)\leq_T u$, so $(i\,\,j)(\overline i\,\,\overline j)\leq_T v$. Since $v\in\NC(D_n,c)$ is noncrossing, it has at most $1$ balanced cycle other than $(\overline n\,\,n)$. This means that $i$ and $j$ cannot belong to different balanced cycles of $v$, so they must belong to the same cycle of $v$ by \Cref{Lem1}. 

Now suppose $i=\overline j$. This means $i$ is in a balanced cycle in $u$. We want to show that $i$ is in a balanced cycle of $v$. Suppose, by way of contradiction, that $i$ is in an unbalanced cycle of $v$. Each element of $D_n$ has an even number of balanced cycles, so there is some balanced cycle in $u$ that does not contain $i$; let $k$ be one of the numbers in this other balanced cycle. Since $i$ and $k$ are in different balanced cycles in $u$, we have $(i\,\,k)(\overline i\,\,\overline k)\leq_T u$ by \Cref{Lem1}. Therefore, $(i\,\,k)(\overline i\,\,\overline k)\leq_T v$. Because $i$ is in an unbalanced cycle in $v$, \Cref{Lem1} tells us that $i$ and $k$ are in the same cycle in $v$. Since $\overline k$ is in the same balanced cycle in $u$ as $k$, the same argument we have just given shows that $i$ and $\overline k$ are in the same cycle in $v$. However, this means that the cycle in $v$ containing $i$ contains both $k$ and $\overline k$, so it is balanced. This is a contradiction.     
\end{proof}

Our goal in the remainder of this section is to prove \Cref{thm:main2}. The Coxeter number of $D_n$ is $2n{-}2$. We already know by \Cref{prop:forward_cinv} that $O_{\pop_T}(c^{-1})$ is a forward orbit of $\pop_T:D_n\to D_n$ of size $2n{-}2$. Thus, we need to prove that $\pop_T^{2n-3}(u_0)=e$ for all $u_0\in D_n$. 

Fix $u_0\in D_n$. For $k\geq 0$, let $u_k=\pop_T^k(u_0)$ and $v_k=\pi_T(u_k)$. Thus, $u_{k+1}=u_kv_k^{-1}$. Our goal is to show that $u_{2n-3}=e$. Since $\pop_T^{-1}(e)=\NC(D_n,c)$, it suffices to show that $u_{2n-4}\in\NC(D_n,c)$. The following lemma will do all of the heavy lifting. When we refer to ``the circle,'' we mean the circle that has $2n{-}2$ equally-spaced points labeled with the numbers $\overline 1,\overline 2,\ldots,\overline{n-1},1,2,\ldots,n-1$ in clockwise order and that has its center labeled with the numbers $\overline n$ and $n$. Define the \defn{predecessor} of an element $j\in\pm[n-1]$ to the be number that appears immediately before $j$ in the clockwise order of the circle. For example, the predecessor of $2$ is $1$, and the predecessor of $\overline 1$ is $n-1$.  

\begin{lemma}\label{lem:DMainLemma}
Suppose $j\in\pm[n-1]$ is such that $u_0^{-1}(j)\not\in\{\overline n,n\}$. Then $u_{2n-4}^{-1}(j)$ is either $j$ or the predecessor of $j$. 
\end{lemma}

\begin{proof}
Define a total order $\prec$ on $\pm[n-1]$ by $\overline 1\prec\cdots\prec\overline{n-1}\prec 1\prec\cdots\prec n-1$. Define $\prec_j$ to be the total order on $\pm[n-1]$ obtained by cyclically shifting the total order $\prec$ so that the maximum element in the order $\prec_j$ is $j$. In other words, $\prec_j$ is given by reading the numbers clockwise around the circle so that the last number read is $j$. Given a set $Q\subseteq\pm[n-1]$, we write $i\prec_j Q$ (respectively, $Q\prec_J i$) to mean that $i\prec_j q$ (respectively, $q\prec_j i$) for all $q\in Q$. 

If $u_k^{-1}(j)=j$ for some $0\leq k\leq 2n-4$, then $u_{k'}^{-1}(j)=j$ for all $k'\geq k$. In particular, $u_{2n-4}^{-1}(j)=j$ in this case. Thus, we may assume $u_k^{-1}(j)\neq j$ for all $0\leq k\leq 2n-4$.  

Suppose $k\in\{0,\ldots,2n-4\}$ is such that $u_k^{-1}(j)\not\in\{\overline n,n\}$ and $u_{k+1}^{-1}(j)=v_k(u_k^{-1}(j))\not\in\{\overline n,n\}$. Let $\mathcal C$ be the cycle in $v_k$ containing $j$. Since $u_k^{-1}(j)$ and $j$ are in the same cycle in $u_k$, \Cref{Lem2} tells us that $u_k^{-1}(j)\in\mathcal C$. Because $v_k$ is a noncrossing partition, the cycle $\mathcal C$ is ordered clockwise. Consequently, $u_{k+1}^{-1}(j)$ is the element of $\mathcal C$ that occurs right after $u_k^{-1}(j)$ in the clockwise cyclic order. Since $u_k^{-1}(j)\neq j$, we must have $u_k^{-1}(j)\prec_j u_{k+1}^{-1}(j)$.

If there does not exist an integer $i\in\{0,\ldots,2n-4\}$ such that $u_i^{-1}(j)\in\{\overline n,n\}$, then it follows from the preceding paragraph that $u_0^{-1}(j)\prec_j u_1^{-1}(j)\prec_j\cdots\prec_j u_{2n-4}^{-1}(j)$. We have assumed $u_{2n-4}^{-1}(j)\neq j$, so it follows that $u_{2n-4}^{-1}(j)$ is the predecessor of $j$, as desired. 

We may now assume there exists $i\in\{0,\ldots,2n-4\}$ with $u_i^{-1}(j)\in\{\overline n,n\}$. We will actually show that this case cannot occur because it contradicts our assumption that $u_{k}^{-1}(j)\neq j$ for all $0\leq k\leq 2n-4$. Since $u_0^{-1}(j)\not\in\{\overline n,n\}$ by the hypothesis of the lemma, we must have $i\geq 1$ and $u_{i-1}^{-1}(j)\not\in\{\overline n,n\}$. Let $\mathcal C_{i-1}$ and $\mathcal C_i$ be the cycles in $v_{i-1}$ and $v_i$, respectively, that contain $j$. 

Let $Y$ be the set consisting of the next $n-1$ numbers after $u_{i-1}^{-1}(j)$ in clockwise order around the circle. Because $v_{i-1}(u_{i-1}^{-1}(j))=u_i^{-1}(j)\in\{\overline n,n\}$, we can use \Cref{Lem3} to see that none of the elements of $Y$ lie in the same cycle of $v_{i-1}$ as $u_{i-1}^{-1}(j)$. Since $j$ and $u_{i-1}^{-1}(j)$ are in the same cycle of $u_{i-1}$, we know by \Cref{Lem2} that $j$ and $u_{i-1}^{-1}(j)$ are in the same cycle of $v_{i-1}$. In other words, $u_{i-1}^{-1}(j)\in\mathcal C_{i-1}$, and $\mathcal C_{i-1}\cap Y=\emptyset$. Hence, $j\not\in Y$. This implies that $u_{i-1}^{-1}(j)\prec_j Y\prec_j j$. In particular, $u_{i-1}^{-1}(j)$ must be among the smallest $n-2$ numbers in the order $\prec_j$. We also know that $u_i^{-1}(j)\in\mathcal C_{i-1}$ because $u_i^{-1}(j)=v_{i-1}(u_{i-1}^{-1}(j))$. Since $u_i^{-1}(j)\in\{\overline n,n\}$, this implies that $\mathcal C_{i-1}$ is an unbalanced cycle in $v_{i-1}$. In fact, this shows that $v_{i-1}$ has no balanced cycles (if it did, then by the definition of the bijection $g$, one of the balanced cycles of $v_{i-1}$ would be $(\overline n\,\, n)$).  

Athanasiadis and Reiner \cite{AthanasiadisReiner} gave an explicit combinatorial description of the covering relations in $\NC(D_n,c)$; it follows from that description that if $x,y\in\NC(D_n,c)$ are such that $x\leq_T y$, then every unbalanced cycle in $y$ is a union (as a set) of unbalanced cycles in $x$. According to \Cref{lem:monotone}, we have $v_i\leq_T v_{i-1}$. Because $v_{i-1}$ has no balanced cycles, neither does $v_i$. Furthermore, $\mathcal C_i\subseteq \mathcal C_{i-1}$.  Since $j$ and $u_i^{-1}(j)$ are in the same cycle in $u_i$, we have $(j\,\,u_i^{-1}(j))(\overline j\,\,\overline{u_i^{-1}(j)})\leq_T u_i$ by \Cref{Lem1}. Therefore, $(j\,\,u_i^{-1}(j))(\overline j\,\,\overline{u_i^{-1}(j)})\leq_T v_i$. Since $v_i$ has no balanced cycles, we can use \Cref{Lem1} again to see that $j$ and $u_i^{-1}(j)$ are in the same unbalanced cycle in $v_i$. That is, $u_i^{-1}(j)\in\mathcal C_i$. Now, $u_{i+1}^{-1}(j)=v_i(u_i^{-1}(j))$ is the element of $\mathcal C_i$ that occurs immediately after $u_i^{-1}(j)$ in the clockwise cyclic order, so $Y\prec_j u_{i+1}^{-1}(j)$. Hence, $u_{i+1}^{-1}(j)$ is among the largest $n-2$ elements of $\pm[n-1]$ in the order $\prec_j$.  

We have shown that if $i\in\{1,\ldots,2n-4\}$ is such that $u_{i}^{-1}(j)\in\{\overline n,n\}$ and $u_{i-1}^{-1}(j)\not\in\{\overline n,n\}$, then $u_{i-1}^{-1}(j)$ is among the smallest $n-2$ elements of $\pm[n-1]$ in the order $\prec_j$ while $u_{i+1}^{-1}(j)$ is among the largest $n-2$ elements. We have also seen that if $k\in\{0,\ldots,2n-4\}$ is such that $u_k^{-1}(j)\not\in\{\overline n,n\}$, then $u_k^{-1}(j)\prec_j u_{k+1}^{-1}(j)$. We assumed in the hypothesis of the lemma that $u_0^{-1}(j)\not\in\{\overline n,n\}$. It follows from these facts that $u_k^{-1}(j)\not\in\{\overline n,n\}$ for every $k\in\{0,\ldots,2n-4\}\setminus\{i\}$. Therefore, the sequence $u_0^{-1}(j),u_1^{-1}(j),\ldots,u_{i-1}^{-1}(j),u_{i+1}^{-1}(j),\ldots,u_{2n-4}^{-1}(j)$ is strictly increasing in the order $\prec_j$. However, this sequence does not include any of the $n-1$ elements of $Y$ because $u_{i-1}^{-1}(j)\prec_j Y\prec_j u_{i+1}^{-1}(j)$. This is our desired contradiction. 
\end{proof}

\begin{proof}[Proof of \Cref{thm:main2}]
Earlier, we chose an arbitrary $u_0\in D_n$ and let $u_k=\pop_T^k(u_0)$. Our goal is to prove that $u_{2n-3}=e$; in order to do so, it suffices to show that $u_{2n-4}\in\NC(D_n,c)$. Let $m=u_0(n)$. \Cref{lem:DMainLemma} tells us that if $j\in\pm[n-1]\setminus\{\overline m,m\}$, then $u_{2n-4}^{-1}(j)$ is either $j$ or the predecessor of $j$. We now consider several cases. 


\medskip

\noindent {\bf Case 1.} Suppose $u_{2n-4}^{-1}(m)=m$. Then $u_{2n-4}^{-1}(\overline m)=\overline m$. If $m\not\in\{\overline n,n\}$, then one can straightforwardly deduce from \Cref{lem:DMainLemma} that $u_{2n-4}=e\in\NC(D_n,c)$. If $m\in\{\overline n,n\}$, then one can use \Cref{lem:DMainLemma} to see that $u_{2n-4}$ is either $e$ or $c$; in either case, $u_{2n-4}\in\NC(D_n,c)$. 

\medskip

\noindent {\bf Case 2.} Suppose that $u_{2n-4}^{-1}(m)\not\in\{\overline n,n,m\}$ and that $\overline m, u_{2n-4}^{-1}(m),m$ appear in this (cyclic) order when we read the numbers on the circle in clockwise order. Let $Z$ be the set of numbers that either appear between $\overline m$ and $u_{2n-4}^{-1}(m)$ (including $\overline m$ and $u_{2n-4}^{-1}(m)$) or appear between $m$ and $\overline{u_{2n-4}^{-1}(m)}$ (including $m$ and $\overline{u_{2n-4}^{-1}(m)}$) in the clockwise order on the circle. Using \Cref{lem:DMainLemma}, it is straightforward to check that $u_{2n-4}$ cyclically permutes the elements of $Z$ in clockwise order and fixes each element of $\pm[n-1]\setminus Z$ (see \Cref{FigTsack5}). Since $u_{2n-4}\in D_n$, we must have $u_{2n-4}(n)=\overline n$ and $u_{2n-4}(\overline n)=n$. Therefore, $u_{2n-4}=g(\rho)$, where $\rho$ is the type-$D_n$ noncrossing set partition whose blocks are $Z\cup\{\overline n,n\}$ and all of the singleton subsets of $\pm[n-1]\setminus Z$. Thus, $u_{2n-4}\in\NC(D_n,c)$. 

\medskip

\noindent {\bf Case 3.} Suppose that $u_{2n-4}^{-1}(m)\not\in\{\overline n,n,m\}$ and that $m, u_{2n-4}^{-1}(m),\overline m$ appear in this (cyclic) order when we read the numbers on the circle in clockwise order. Let $B$ be the set of numbers that appear between $m$ and $u_{2n-4}^{-1}(m)$ (including $m$ and $u_{2n-4}^{-1}(m)$) in the clockwise order on the circle. Using \Cref{lem:DMainLemma}, it is straightforward to check that $u_{2n-4}$ cyclically permutes the elements of $B$ in clockwise order, cyclically permutes the elements of $\overline B$ in clockwise order, and fixes each element of $\pm[n-1]\setminus (B\cup(\overline B))$ (see \Cref{FigTsack5}). Since $u_{2n-4}$ is in $ D_n$, it must fix $n$ and $\overline n$. Therefore, $u_{2n-4}=g(\rho)$, where $\rho$ is the type-$D_n$ noncrossing set partition whose blocks are $B$, $\overline B$, and all of the singleton subsets of $\pm[n]\setminus (B\cup(\overline B))$. Thus, $u_{2n-4}\in\NC(D_n,c)$.

\medskip

\noindent {\bf Case 4.} Suppose that $u_{2n-4}^{-1}(m)\in\{\overline n,n\}$ and either that $u_{2n-4}^{-2}(m)=m$ or that the numbers $m,u_{2n-4}^{-2}(m),\overline m$ are distinct and appear in this (cyclic) order when we read the numbers on the circle in clockwise order. Let $B$ be the set of numbers that appear between $m$ and $u_{2n-4}^{-2}(m)$ (including $m$ and $u_{2n-4}^{-1}(m)$) in the clockwise order on the circle. Using \Cref{lem:DMainLemma}, it is straightforward to check that $u_{2n-4}$ cyclically permutes the elements of $B\cup\{u_{2n-4}^{-1}(m)\}$ in clockwise order, cyclically permutes the elements of $\overline B\cup\{\overline u_{2n-4}^{-1}(m)\}$ in clockwise order, and fixes each element of $\pm[n-1]\setminus (B\cup(\overline B))$ (see \Cref{FigTsack5}). Therefore, $u_{2n-4}=g(\rho)$, where $\rho$ is the type-$D_n$ noncrossing set partition whose blocks are $B\cup\{u_{2n-4}^{-1}(m)\}$, $\overline B\cup\{\overline{u_{2n-4}^{-1}(m)}\}$, and all of the singleton subsets of $\pm[n-1]\setminus (B\cup(\overline B))$. Thus, $u_{2n-4}\in\NC(D_n,c)$. 

\medskip 

\noindent {\bf Case 5.} Suppose that $u_{2n-4}^{-1}(m)\in\{\overline n,n\}$ and either that $u_{2n-4}^{-2}(m)=\overline m$ or that the numbers $\overline m,u_{2n-4}^{-2}(m),m$ are distinct and appear in this (cyclic) order when we read the numbers on the circle in clockwise order. Let $r_1,r_2,\ldots,r_q$ be the numbers that appear between $\overline m$ and $u_{2n-4}^{-2}(m)$ (including $\overline m$ and $u_{2n-4}^{-2}(m)$) in the clockwise order on the circle. In particular, $r_1=\overline m$ and $r_q=u_{2n-4}^{-2}(m)$. In this case, one can use \Cref{lem:DMainLemma} to show that $u_{2n-4}$ is the single cycle \[(r_1\,\,r_2\cdots r_q\,\, u_{2n-4}^{-1}(m)\,\,\overline {r_1}\,\,\overline {r_2}\cdots \overline{r_q}\,\,\overline{u_{2n-4}^{-1}(m)})\] (meaning $u_{2n-4}$ fixes every element of $\pm[n]$ that is not in this cycle). However, this means that $u_{2n-4}$ has a single balanced cycle, contradicting the fact that every element of $D_n$ has an even number of balanced cycles. Therefore, this case cannot even occur.  
\end{proof}

\begin{figure}[htbp]
\begin{center} 
\includegraphics[height=4cm]{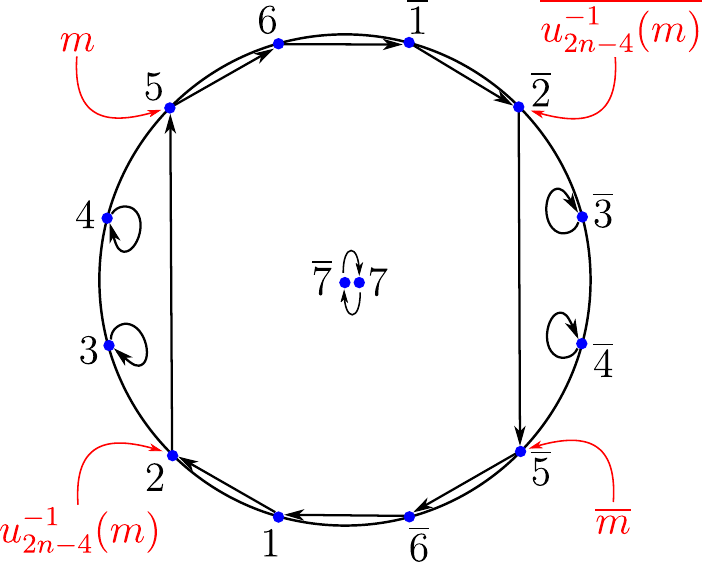}\qquad\qquad\includegraphics[height=4cm]{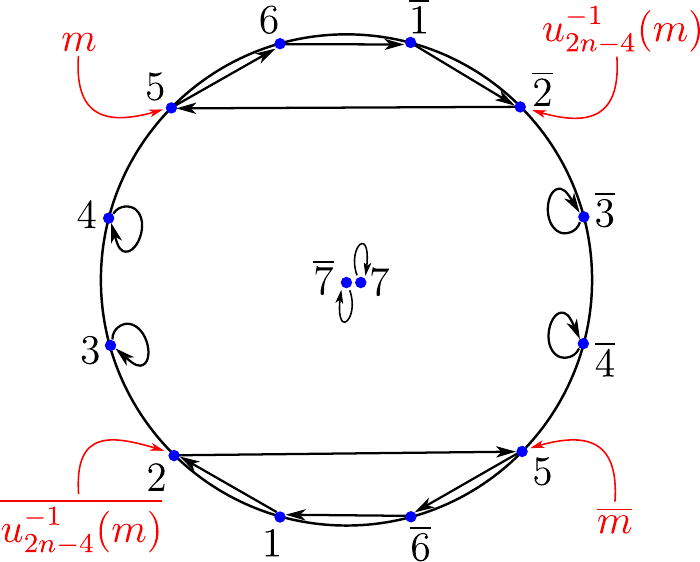}\end{center}
\vspace{1cm}
\begin{center}
\includegraphics[height=4cm]{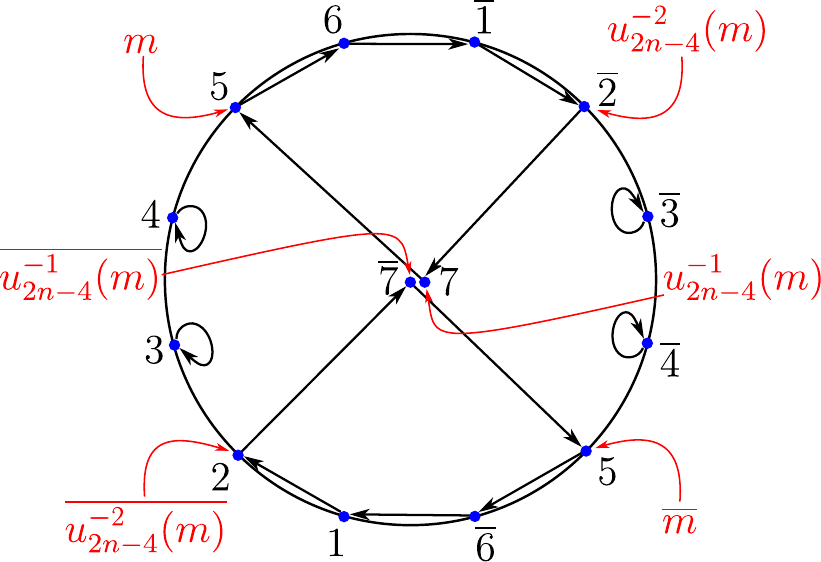}\end{center}
\caption{Diagrams illustrating Case~2 (upper left), Case~3 (upper right), and Case~4 (bottom middle) from the proof of \Cref{thm:main2}. In each example, we have $n=7$ and $m=5$.}\label{FigTsack5}
\end{figure}


\begin{conjecture}
\label{conj:D}
The number of elements of $D_n$ that require exactly $2n-3$ iterations of $\pop_T$ to reach the identity is $n (2^{n-1}-2)+1$.
\end{conjecture}

\section{Exceptional Types}

In this section, we conclude our study of $\pop_T$ on finite irreducible Coxeter groups by considering the exceptional types $E_6$, $E_7$, $E_8$, $F_4$, and $H_4$.  Recall that an orbit $O_{\pop_T}(x)$ is called \defn{periodic} if $\pop_T^i(x)=x$ for some $i>0$, and that for a Weyl group, $c$ is conjugate to $c^k$ for any $k$ relatively prime to $h$~\cite{springer1974regular}.

\begin{theorem}
\label{thm:periodic}
The map $\pop_T$ has periodic orbits of size $h$ for $W$ of type $E_6$, $E_7$, $E_8$, $F_4$, or $H_4$.  Furthermore, for any $k\not\equiv \pm 1 \pmod h$ coprime to $h$ such that $c$ is conjugate to $c^k$ in $W$, the set \[O_k = \{w \in W : c^w = c^k\}\] is a periodic orbit of size $h$, where $c^w = w^{-1}cw$.
\end{theorem}

\begin{proof}
We computed the entire orbit structure of $\pop_T$ on groups of types $F_4$, $H_4$, and $E_6$ using a combination of Sage~\cite{sagemath} and GAP3 with CHEVIE~\cite{Sch97,GH96}.  This data is summarized in~\Cref{fig:data}.  We confirmed by computer that the two periodic orbits of $F_4$ (each of size $h=12$) lift via~\Cref{sec:folding} to the two periodic orbits of $E_6$ (each again of size $h=12$).  In $F_4$ and $E_6$, the two periodic orbits are $O_5$ and $O_7$---and there are no others.
  

The number of elements requiring $i$ iterations of $\pop_T$ to reach the identity in $H_4$ is not included in~\Cref{fig:data} for reasons of space; but for $i=0,1,\ldots,34$ these numbers are 1, 279, 467, 465, 663, 690, 675, 660, 750, 660, 495, 390, 390, 390, 390, 390, 390, 390, 390, 390, 375, 375, 375, 375, 375, 300, 300, 300, 225, 75, 66, 66, 33, 30, 15, with an additional $1800$ elements in 60 periodic orbits, each of size $h=30$.  Both $O_{11}$ and $O_{19}$ are periodic orbits in $H_4$ (and $c$ is not conjugate to $c^k$ for $k=7,13,17,23$).

In $E_7$ with $h=30$, we checked that $O_k$ for $k=5,7,11,13$ all form periodic orbits, but we did not compute the full orbit structure of $\pop_T$.

To show that $E_8$ also has periodic orbits, we can use the unfolding of $H_4$ to $E_8$, which necessarily takes periodic orbits to periodic orbits by \Cref{thm:pop_embedding}.  Write $t_{1^{a_1}\cdots 8^{a_8}}$ for the reflection orthogonal to the root $a_1\alpha_1+\cdots+a_8\alpha_8$. For example, $t_{123^24^25^2678}$ is the reflection orthogonal to $\alpha_{1} + \alpha_{2} + 2\alpha_{3} + 2\alpha_{4} + 2\alpha_{5} + \alpha_{6} + \alpha_{7} + \alpha_{8}$. The $E_8$ element \[t_2 \cdot t_{123456} \cdot t_3 \cdot t_{1234^25678} \cdot t_{123^24^25^2678} \cdot t_{134567}\] lies in a periodic orbit of size $h=30$ and is the unfolding of the $H_4$ element with reduced word (in simple reflections) \[s_3s_4s_3s_2s_1s_2s_1s_3s_2s_1s_4s_3s_2s_1s_2s_3s_4.\]
With computational assistance from C.~Stump, we confirmed that the sets $O_k$ for $k \in \{7,11,13,17,19,23\}$ all form periodic orbits under $\pop_T$.\end{proof}  

It would be interesting to fully characterize which elements of $W$ give rise to periodic orbits.





\section{Normal Forms}

\subsection{Dual braid lifts}
\label{sec:lifts}
For this subsection, fix $W$ to be a finite irreducible Coxeter group with no periodic orbits under $\pop_T$ other than that of the identity element.  We would like to provide a canonical lift of $W$ to Bessis' dual braid monoid, just as $W$ has a standard lift (using reduced words in simple reflections) to its positive braid monoid.

Recall that we introduced $\pop_T$ as the dual version of $\pop_S$.  As $\pop_S$ is closely related to Brieskorn normal form, we consider the analogous factorization coming from $\pop_T$---this succeeds because each image under $\pi_T$ is a noncrossing partition and hence has a canonical lift to the dual braid monoid.

\begin{theorem}
\label{thm:normal_form}
Let $W$ be a finite irreducible Coxeter group whose only periodic orbit under $\pop_T$ is that of the identity.  For $w\in W$, write $w_1=w$ and $w_i=\pop_T(w_{i-1},c)$ for $i\geq 2$. Let $k=\lvert O_{\pop_T}(w)\rvert-1$ so that (assuming $w\neq e$) $w_k\neq w_{k+1}=e$.  Then the factorization \[w=\pi_T(w_k,c)\pi_T(w_{k-1},c)\cdots\pi_T(w_1,c)\] gives a lift from $W$ to the dual braid monoid of $W$.
\end{theorem}




\begin{example}
In type $A_5$ with $c=(123456)$ and $w=w_1=(135642)$, we have
\begin{align*}
\pi_T(w_1)&= (123456) &\text{ and } &&w_2=\pop_T(w_1)&=(12634),\\
\pi_T(w_2)&= (12346) &\text{ and } &&w_3=\pop_T(w_2)&=(246),\\
\pi_T(w_3)&= (246) &\text{ and } &&w_4=\pop_T(w_3)&=e,
\end{align*}
so that
\[w=(246) \cdot (12346) \cdot (123456),\]
which may be interpreted as a product of simple elements in the dual braid monoid.
\end{example}

\begin{remark}
T.~Douvropoulos has suggested that one could extend our construction to a general finite complex reflection group $W$ by using noncrossing reflections as a generating set for $W$, with a noncrossing projection of $w \in W$ defined as the join in the noncrossing partition lattice of those noncrossing reflections appearing in a shortest length word for $w$.  Explicit computations show that $G_5$ has periodic orbits not containing the identity under these definitions.
\end{remark}

\subsection{Stabilized-interval-free elements}

A permutation $w\in\mathfrak S_n$ is called \defn{stabilized-interval-free} (SIF) if it does not stabilize a proper subinterval of $[n]$; that is, there do not exist $i,j\in[n]$ with $1\leq i\leq j\leq n$ satisfying $w(\{i,\ldots,j\})=\{i,\ldots,j\}$ and $\{i,\ldots,j\}\neq[n]$. These permutations were initially investigated from an enumerative point of view in \cite{callan2004counting} (see also \cite{blitvic2014stabilized}). They are also important in the theory of the totally nonnegative Grassmannian because they parameterize connected positroids \cite{ArdilaRinconWilliams}. Using the combinatorial description of noncrossing partitions in type~$A$ discussed in \Cref{sec:typeab}, it follows that $w$ is SIF if and only if $\pi_T(w,c)=c$, where $c$ is the Coxeter element $(12\cdots n)$. 

Now consider an arbitrary finite irreducible Coxeter group $W$ with a fixed Coxeter element $c$.  The following construction naturally generalizes the notion of a stabilized-interval-free permutation. 


\begin{definition}
\label{def:sif}
An element $w\in W$ is \defn{stabilized-interval-free} (SIF) if $\pi_T(w,c)=c$.
\end{definition}


For $w\in W$, let $P$ be the smallest parabolic subgroup containing $w$. We call the irreducible components $P_1,\ldots,P_k$ of $P$ the \defn{blocks} of $w$. Note that $P=P_1 \times P_2 \times \cdots \times P_k$.  The following is immediate. 

\begin{proposition}
Every $w \in W$ can be written uniquely as a product of SIF elements in the blocks of $\pi_T(w,c)$.
\end{proposition}

\begin{figure}[htbp]
\begin{center}
\begin{tabular}{c|ccccc|c}
rank $n$ & 2 & 3 & 4 & 5 & 6 & OEIS \\ \hline
$A_n$ & 2 & 7 & 34 & 206 & 1476 & \href{https://oeis.org/A075834}{A075834}\\
$B_n$ & 3 & 20 & 179 & 1944 & 24674 &\\
$D_n$ &     & 7 & 74 & 891 &12004 &\\
$E_6$ & & & & & 33610 &\\
$F_4$ & & &  762 & & &\\
$G_2$ & 5 & & & & &\\
$H_3$& & 69 & & & &\\
$H_4$& & & 12802 & & &\\
$I_2(m)$ & $m-1$ & & & & &\\
\end{tabular}
\end{center}
\caption{Numbers of SIF elements in finite irreducible Coxeter groups of small rank.}
\end{figure}


\section*{Acknowledgements}
N.W. was partially supported by a Simons Foundation Collaboration Grant. C.D. was supported by a Fannie and John Hertz Foundation Fellowship and an NSF Graduate Research Fellowship (grant number DGE-1656466).  We thank C.~Stump for his valuable computational assistance in~\Cref{thm:periodic} (especially for $E_8$) and T.~Douvropoulos for many insightful comments on an early draft.  N.W. thanks T.~Gobet for a fruitful collaboration and conversation in 2012 which led to the discussion in~\Cref{sec:lifts}.

\bibliographystyle{amsalpha}
\bibliography{reflection_sorting}

\end{document}